\documentclass{amsart}
\usepackage{amssymb}
\usepackage{nicefrac}
\usepackage{ytableau}
\usepackage{tikz}
\usepackage{mfabacus}
\usepackage{tikz-3dplot}
\usetikzlibrary{shapes}

\title{A further generalisation of bar-core partitions}
\author[D. Yates]{Dean Yates}
\address{Queen Mary, University of London\\
Mile End Road\\
Bethnal Green\\
London\\
E1 4NS}
\email{d.yates@qmul.ac.uk}
\date{01 June 2021} 

\parindent=\normalparindent

\def\ps@firstpage{%
  \ps@plain
  \def\@oddfoot{{\scriptsize    
      ISSN\string: \@Eissn\hfill\url{\journalURL}
}}
  \let\@evenfoot\@oddfoot
  \def\@oddhead{\@serieslogo\hss}%
  \let\@evenhead\@oddhead 
}

\newtheorem{thm}{Theorem}[section]
\newtheorem{lem}[thm]{Lemma}
\newtheorem{prop}[thm]{Proposition}
\newtheorem{cor}[thm]{Corollary}

\theoremstyle{definition}
\newtheorem*{defn}{Definition}
\newtheorem*{exam}{Example}

\theoremstyle{remark}
\newtheorem*{rema}{Remark}

\begin{document}

\begin{abstract}
When $p$ and $q$ are coprime odd integers no less than 3, Olsson proved that the $q$-bar-core of a $p$-bar-core is again a $p$-bar-core. We establish a generalisation of this theorem: that the $p$-bar-weight of the $q$-bar-core of a bar partition $\lambda$ is at most the $p$-bar-weight of $\lambda$. We go on to study the set of bar partitions for which equality holds and show that it is a union of orbits for an action of a Coxeter group of type $\tilde C_{\nicefrac{(p-1)}{2}}\times\tilde C_{\nicefrac{(q-1)}{2}}$. We also provide an algorithm for constucting a bar partition in this set with a given $p$-bar-core and $q$-bar-core. 
\end{abstract}

\maketitle

\tableofcontents

\section{Introduction}
\noindent
The study of the set of core partitions, those partitions whose Young diagrams are without $s$-hooks for some natural number $s$, has revealed a great deal about the representation theory of the symmetric groups: when $s$ is a prime, the set of $s$-cores is an index for the $s$-blocks of all symmetric groups of a given defect, and the relationships between these blocks are dictated by the combinatorics of $s$-cores. The purpose of this paper is to establish analogues of the results in Fayers' `A generalisation of core partitions' \cite{F1} for \textit{bar partitions}, i.e. partitions with distinct parts (or simply finite subsets of $\mathbb{N}$). Although the definitions differ, the ideas in \cite{F1} translate very well to the notion of $p$\textit{-bar-cores}, including a generalisation of a result established by Olsson in \cite{O2}: the $q$-bar-core of a $p$-bar-core is a $p$-bar-core. The motivation behind the study of bar partitions is their correspondence to \textit{projective}, or \textit{spin representations} of the symmetric group \cite{HH}. The aforementioned generalisation of Olsson's theorem confirms a relationship between these representations.

We will begin with a few definitions, which will seem very familiar to those acquainted with the representation theory of the symmetric group, that have been adapted to suit our purposes. Using some basic results, we consider an action of $\mathfrak{W}_p$, the Weyl group of type $\tilde{C}_{(p-1)/2}$, on the set of bar partitions $\mathcal{P}_2$, and discover some interesting symmetry. We then consider the problem of constructing the smallest bar partition with a given $p$-bar-core and $q$-bar-core, for coprime odd $p,q\geq3$. We finish by investigating the orbits of the Yin and Yang partitions \cite{BO} under the group action of $\mathfrak{W}_p\times\mathfrak{W}_q$. 

\section{Definitions}
\noindent
A \textbf{bar partition} $\lambda\in\mathcal{P}_2$ is a decreasing sequence of distinct positive integers (often referred to as a 2-regular or strict partition). For odd integers $p\geq3$, the $p$\textbf{-runner abacus} \cite{F2} has $p$ infinite vertical runners numbered from left to right $\nicefrac{(p+1)}{2},\nicefrac{(p+3)}{2},\dots,p-1,0,1,\dots,\nicefrac{(p-1)}{2}$, with the positions on runner $i$ labelled with the integers with $p$ residue $i$, increasing down the runner, so that position $x+1$ appears directly to the right of position $x$ (for $x\not\equiv\nicefrac{(p-1)}{2}$ mod $p$). We obtain a visual representation of $\lambda$ on the $p$-runner abacus by placing a bead on position $x$ for each $x\in\lambda$ and each integer $x<0$ such that $-x\not\in\lambda$; position 0 remains empty. 

\begin{exam}
The bar partition $(9,8,7,5,3)$ has the following bead configuration on the 5-runner abacus (we indicate the zero position with a white bead): 
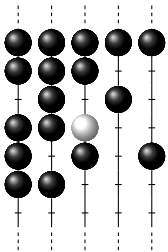
\begin{center}\abacus(vvvvv,bbbbb,bbbnn,nbnbn,bbonn,bnbnb,bbnnn,nnnnn,vvvvv)\end{center}
\end{exam}

Let $\mathcal{A}(\lambda)$ denote the set containing all integers that label positions occupied by beads in the bead configuration for $\lambda\in\mathcal{P}_2$ on the $p$-runner abacus: 
$$x\in\mathcal{A}(\lambda)\Leftrightarrow\begin{cases}x\in\lambda,&x>0\text{, or}\\-x\not\in\lambda,&x<0.\end{cases}$$
Note that this is independent from the choice of $p$. 
\newline
For odd integers $p\geq3$, removing a $p$\textbf{-bar} from $\lambda\in\mathcal{P}_2$ means either \newline
(i) removing $x\in\lambda$ such that $0\leq x-p\not\in\lambda$, and replacing $x$ with $x-p$ if $x\neq p$; or \newline
(ii) removing two parts $x,p-x\in\lambda$ (where $0<x<p$).
\newline
 ($p$ must be odd because of the incompatible possibility that a bar partition could have a $2p$-bar but not a $p$-bar, e.g. $p=4$ and the partition $(6,2)$.) \newline
In terms of the abacus, removing a $p$-bar from $\lambda$ corresponds to moving a bead at position $x$ to position $x-p$ (replacing $x\in\mathcal{A}(\lambda)$ with $x-p$), then moving the bead at position $p-x$ to position $-x$ (replacing $p-x\in\mathcal{A}(\lambda)$ with $-x$). 
\newline
When it is not possible to remove any $p$-bars from $\lambda$, i.e. when $x-p\in\mathcal{A}(\lambda)$ for all $x\in\mathcal{A}(\lambda)$, we say that $\lambda$ is a $p$\textbf{-bar-core}, and we denote the set of $p$-bar-cores by $\overline{C}_p$. 
\newline
Since removing a $p$-bar always corresponds to moving beads up on their runners to unoccupied positions, we have reached the bead configuration of a $p$-bar-core when all beads are moved up their runners as far as possible. The order in which these moves are made is irrelevant; we always end up at the same bead configuration. Hence we may define \textbf{the} $p$\textbf{-bar-core} of a bar partition $\lambda$, which we denote by $\overline{\lambda}_p$. The number of $p$-bars which can be successively removed from $\lambda$ is the $p$\textbf{-bar-weight} of $\lambda$, and we denote this quantity by $\overline{\text{wt}}_p(\lambda)$; denoting by $|\mu|$ the sum of the parts of the bar partition $\mu$, 
$$\overline{\text{wt}}_p(\lambda):=\tfrac{|\lambda|-|\overline{\lambda}_p|}{p}.$$ 
The number of bead moves needed to reach the bead configuration for $\overline{\lambda}_p$ from the bead configuration for $\lambda$ is equal to twice the $p$-bar-weight of $\lambda$, because removing a $p$-bar corresponds to two moves of the beads. The $p$-bar-weight of $\lambda$ is therefore equal to half the number of pairs $(x,a)\in\mathcal{A}(\lambda)\times\mathbb{N}$ such that $x-ap\not\in\mathcal{A}(\lambda)$. 

\begin{exam}
The $5$-bar-core of the bar partition from the previous example is $\overline{(9,8,7,5,3)}_5=(4,3)$, and has the following bead configuration on the $5$-runner abacus: 
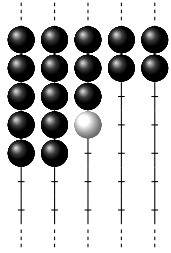
\begin{center}\abacus(vvvvv,bbbbb,bbbbb,bbbnn,bbonn,bbnnn,nnnnn,nnnnn,vvvvv)\end{center}
\end{exam}

We are now equipped with tools analogous to those James used in his seminal book on the representation theory of the symmetric group via the combinatorics of (not necessarily strict) partitions \cite{JK}. For the benefit of readers unfamiliar with James' work, we outline the theory of \textit{rim-hooks} and \textit{cores} here.

We may visually represent a partition $\alpha=(\alpha_1,\alpha_2,\dots,\alpha_r)$, i.e. a decreasing sequence of (not necessarily distinct) positive integers $\alpha_1\geq\alpha_2\geq\dots\geq\alpha_r$, by its \textbf{Young diagram} $[\alpha]$, which has $\alpha_i$  nodes in the $i^\text{th}$ row, for each $i\in\{1,\dots,r\}$, with each row starting in the first column. The \textbf{hook-length} $h_{i,j}$ of the $(i,j)$-node, in the $i^\text{th}$ row and $j^\text{th}$ column of $[\alpha]$, is found by adding the number of $(k,j)$-nodes with $k\geq i$ to the number of $(i,l)$-nodes with $l>j$. We refer to the $(i,j)$-nodes with $(i+1,j+1)\not\in[\alpha]$ as the \textbf{rim} of $[\alpha]$. The $h_{i,j}$ pairwise adjacent nodes along the rim of $[\alpha]$ from the lowest node in the $j^\text{th}$ column, i.e. the $(k,j)$-node with $k$ maximal, to the $(i,\alpha_i)$-node are collectively called a \textbf{rim} $h_{i,j}$-\textbf{hook}. Whenever a diagram $[\alpha]$ has an $(i,j)$-node with hook-length $s:=h_{i,j}$, we may remove a rim s-hook from $[\alpha]$ to obtain the Young diagram of a different partition. If instead $[\alpha]$ has no nodes with hook-length $s$, then we say that the partition $\alpha$ is an $s$-\textbf{core}. 

\begin{exam}
Below is the Young diagram $[(4,4,2,1)]$, which has just one 5-hook. The $(1,2)$-node is highlighted with a $\bullet$, the $(1,2)$-hook is highlighted in red, and the removable nodes of the corresponding rim 5-hook are highlighted with $\times$'s:

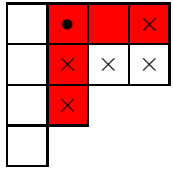
\begin{center}
\ytableaushort{\none \bullet \none \times,\none \times \times \times, \none \times}*{4,4,2,1}*[*(red)]{1+3,1+1,1+1}
\end{center}
\end{exam}

Adopting the convention that $\alpha_i=0$ for each $i$ greater than some fixed $r\in\mathbb{N}$, the strictly decreasing sequence of integers $\alpha_1-1+k>\alpha_2-2+k>\dots$, for some $k\in\mathbb{Z}$, is called a \textbf{beta}-\textbf{set} for the partition $\alpha=(\alpha_1,\alpha_2,\dots)$, and is denoted by $\mathcal{B}^\alpha_k$. \newline
James' $s$\textbf{-abacus} has $s$ runners extending infinitely in both directions, with the leftmost runner labelled by multiples of $s$, and the position directly to the right of $i$ labelled by $i+1$. A bead configuration is associated with a partition $\alpha$ via the beta-set $\mathcal{B}^\alpha:=\mathcal{B}^\alpha_0$ by placing a bead at the position labelled by $\alpha_i-i$ for each $i\in\mathbb{N}$. \newline
Removing a rim $s$-hook from $[\alpha]$ then corresponds to removing an element $x\in\mathcal{B}^\alpha$ such that $x-s\not\in\mathcal{B}^\alpha$, and replacing $x$ with $x-s$. Thus we obtain the bead configuration for an $s$-core by moving the beads in the configuration for $\alpha$ on the $s$-abacus up their runners as far as possible. Since the order in which we move the beads is irrelevant, there is only one $s$-core which can be obtained from a partition $\alpha$ by removing rim $s$-hooks, and we denote \textbf{the} $s$\textbf{-core} of $\alpha$ by $\tilde{\alpha}_s$. The number of moves needed to reach the bead configuration of $\tilde{\alpha}_s$ from the configuration of $\alpha$, or equivalently, the number of rim $s$-hooks which can be removed from the diagram $[\alpha]$, is the $s$-\textbf{weight} of $\alpha$; we denote this quantity by $\text{wt}_s(\alpha)$. \newline
The $s$-\textbf{quotient} of $\alpha$ is the $s$-tuple of partitions corresponding to the bead configuration of each runner of the $s$-abacus as $s$ separate 1-abaci. Each partition $\alpha$ is uniquely determined by its $s$-core $\tilde{\alpha}_s$ and its $s$-quotient. 

\begin{exam}
Below are the configurations of the partition $\alpha:=(4,4,2,1)$ (on the left), and $\tilde{\alpha}_5=(3,1,1,1)$ (on the right), on James' 5-abacus. As noted in the previous example, the 5-weight of $\alpha$ is 1. A beta-set for $\alpha$ is $\mathcal{B}^{\alpha}_0=\{3,2,-1,-3,-5,-6,\dots\}$, we have $\mathcal{B}^{\tilde{\alpha}_5}_0=\{2,-1,-2,-3,-5,-6,\dots\}$, and the 5-quotient of $\alpha$ is $(\varnothing,\varnothing,\varnothing,(1),\varnothing)$ (where $\varnothing$ denotes the empty partition). 
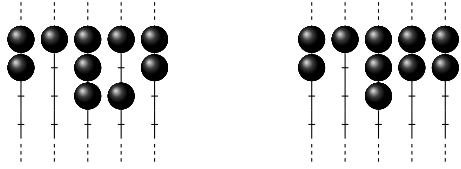
\begin{center}
\abacus(vvvvv,bbbbb,bnbnb,nnbbn,nnnnn,vvvvv)
\hspace{15mm}
\abacus(vvvvv,bbbbb,bnbbb,nnbnn,nnnnn,vvvvv)
\end{center}
\end{exam}

Removing rim $s$-hooks may seem rather arbitrary, but this combinatorial notion has a strong connection with the modular representation theory of the symmetric group, as the two ordinary irreducible represntations corresponding to the partitions $\alpha$ and $\beta$ belong to the same $s$-\textit{block} of $s$-modular irreducible constituents if and only if $\tilde{\alpha}_s=\tilde{\beta}_s$. This important result was first conjectured by Nakayama, and should be referred to as the Brauer-Robinson Theorem after those who first proved it in 1947.

Now that we have seen how the combinatorics of bar partitions is related to James' work, we introduce a way to encode bar partitions, just as the $s$-core and $s$-quotient encode partitions \cite{JK}.

Define the $p$\textbf{-set} \cite{F3} of a bar partition $\lambda$ to be the set $\{\Delta_{i\text{ mod }p}\lambda|i\equiv0,1,\dots,p-1\}$, where $\Delta_{i\text{ mod }p}\lambda$ is the smallest integer $x\equiv i$ modulo $p$ such that $x\not\in\mathcal{A}(\overline{\lambda}_p)$. Since $x\in\mathcal{A}(\overline{\lambda}_p)\Leftrightarrow-x\not\in\mathcal{A}(\overline{\lambda}_p)$, for any bar partition $\lambda$ and $k\not\equiv0$ (mod $p$) we have $\Delta_{k\text{ mod }p}\lambda+\Delta_{-k\text{ mod }p}\lambda=p$, so all of the elements in the $p$-set of any bar partition sum to $\nicefrac{p(p-1)}{2}$. The $p$-set provides a useful way to encode the $p$-bar-core of a bar partition. 

Next we will introduce the $p$\textbf{-quotient} of $\lambda\in\mathcal{P}_2$ \cite{O1}, which is the $p$-tuple 
$$\mathcal{Q}_p(\lambda):=(\lambda^{(0\text{ mod }p)},\dots,\lambda^{(p-1\text{ mod }p)}).$$  
(We will drop the `mod $p$' in our notation for both the $p$-set and $p$-quotient when it is clear which $p$ we are referring to.) \newline
The parts of the bar partition $\lambda^{(0\text{ mod }p)}$ are the elements of the set $\{\nicefrac{x}{p}|x\in\lambda,x\in p\mathbb{Z}\}$. 
\newline
For $j\not\equiv0$ (mod $p$), the $i^\text{th}$ part of the (not necessarily strict) partition $\lambda^{(j\text{ mod }p)}$ is equal to the number of empty spaces above the $i^\text{th}$ lowest bead on runner $j$ in the bead configuration for $\lambda$ on the $p$-runner abacus. 
\newline
It follows from the definition of the $p$-runner abacus that for each $j\not\equiv0$ (mod $p$), the partition $\lambda^{(-j\text{ mod }p)}$ is equal to $(\lambda^{(j\text{ mod }p)})'$, the \textit{conjugate} of the partition $\lambda^{(j\text{ mod }p)}$, the parts of which are the lengths of the columns in the Young diagram for $\lambda^{(j\text{ mod }p)}$. 

\begin{exam}
The bar partition $(9,8,7,5,3)$ has 5-set $\{0,-4,-3,8,9\}$ as its 5-bar-core is $(4,3)$. While the $5$-quotient $\mathcal{Q}_5((4,3))$ contains only empty partitions (as all $p$-quotients of $p$-bar-cores do), the $5$-quotient of $\lambda:=(9,8,7,5,3)$ is $((1),(1),(3),(1,1,1),(1))$. As remarked above, we have $\lambda^{(j\text{ mod }p)}=(\lambda^{(-j\text{ mod }p)})'$ for $j\not\equiv0$ (mod $p$), and this is illustrated below with the Young diagrams of the conjugate partitions $\lambda^{(3\text{ mod }5)}=(1,1,1)$ and $\lambda^{(2\text{ mod }5)}=(3)=(1,1,1)'$, and the self-conjugate partition $\lambda^{(1\text{ mod }5)}=(1)=(1)'=\lambda^{(4\text{ mod }5)}$, which also happens to be the bar partition $\lambda^{(0\text{ mod }5)}$: 
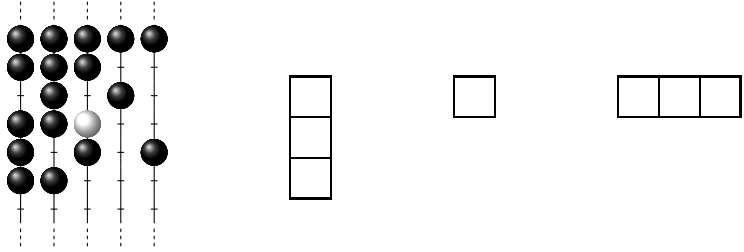
\begin{align*}&\abacus(vvvvv,bbbbb,bbbnn,nbnbn,bbonn,bnbnb,bbnnn,nnnnn,vvvvv)
&&\raisebox{1.8cm}{\ydiagram{1,1,1}}&&\raisebox{1.8cm}{\ydiagram{1}}&&\raisebox{1.8cm}{\ydiagram{3}}\end{align*}
\end{exam}

\section{Properties of the $p$-quotient}
\noindent
These definitions are extremely useful, and Olsson proved that every bar partition is uniquely determined by its $p$-bar-core and $p$-quotient \cite[Proposition 2.2]{O1}.

Later we will utilise the following properties of $\mathcal{Q}_p(\lambda)$. 
\begin{lem}\label{lem3}
Suppose $\lambda\in\mathcal{P}_2$ and $c,p$ are odd postive integers, with $p\geq 3$. 
\begin{enumerate}
\item $\overline{\text{wt}}_{cp}(\lambda)=\overline{\text{wt}}_{c}(\lambda^{(0\text{ mod }p)})+\tfrac{1}{2}\sum_{j=1}^{p-1}\text{wt}_c(\lambda^{(j\text{ mod }p)})$; in particular, $\lambda$ is a $cp$-bar-core if and only if $\lambda^{(0\text{ mod }p)}$ is a $c$-bar-core and every other component of the $p$-quotient of $\lambda$ is a $c$-core. 
\item ${\overline{(\lambda^{(0\text{ mod }p)}})}_c=(\overline{\lambda}_{cp})^{(0\text{ mod }p)}$, and for $j\not\equiv0\:(\text{mod }p)$, the $c$-core of $\lambda^{(j\text{ mod }p)}$ is equal to $(\overline{\lambda}_{cp})^{(j\text{ mod }p)}$.
\end{enumerate}
\end{lem}
\begin{proof}
Removing a $cp$-bar from $\lambda$ means replacing an element $x\in\lambda$ with $x-cp$ (when $x-cp\not\in\mathcal{A}(\lambda)$), then replacing $cp-x\in\mathcal{A}(\lambda)$ with $-x$. \newline
If $x\equiv0$ (mod $p$), then this is equivalent to replacing $\nicefrac{x}{p}\in\mathcal{A}(\lambda^{(0)})$ with $\nicefrac{(x-cp)}{p}=\nicefrac{x}{p}-c$, and replacing $\nicefrac{(cp-x)}{p}\in\mathcal{A}(\lambda^{(0)})$ with $\nicefrac{-x}{p}=\nicefrac{(cp-x)}{p}-c$, i.e., removing a $c$-bar from $\lambda^{(0)}$. \newline
If $x\equiv j$ (mod $p$), then this is equivalent to replacing $\nicefrac{(x-j)}{p}\in\mathcal{B}^{\lambda^{(j)}}_r$ with $\nicefrac{((x-cp)-j)}{p}=\nicefrac{(x-j)}{p}-c$, and then replacing $\nicefrac{((cp-x)-(p-j))}{p}\in\mathcal{B}^{\lambda^{(p-j)}}_r$ with $\nicefrac{(-x-(p-j))}{p}=\nicefrac{((cp-x)-(p-j))}{p}-c$, thus removing a rim $c$-hook from both $\lambda^{(j)}$ and $\lambda^{(p-j)}$.
\end{proof}

\begin{rema}
Note that since $\lambda^{(-j\text{ mod }p)}=(\lambda^{(j\text{ mod }p)})'$ for $j\not\equiv0$ (mod $p$), we can rewrite the first part of the previous lemma as 
$$\overline{\text{wt}}_{cp}(\lambda)=\overline{\text{wt}}_{c}(\lambda^{(0\text{ mod }p)})+\sum_{j=1}^{\tfrac{p-1}{2}}\text{wt}_c(\lambda^{(j\text{ mod }p)}).$$ 
\end{rema}

\section{A level $q$ group action on bar partitions}
\noindent
Now we will consider an action of $\mathfrak{W}_p$, the affine Coxeter group of type $\tilde{C}_{\nicefrac{(p-1)}{2}}$, with generators $\delta_0, \dots, \delta_{\nicefrac{(p-1)}{2}}$ and relations 
\begin{align*}\delta_i^2&=1&&\text{for }0\leq{i}\leq\tfrac{p-1}{2},\\
\delta_i\delta_j&=\delta_j\delta_i&&\text{for }0\leq{i}<j-1\leq\tfrac{p-3}{2},\\
\delta_i\delta_{i+1}\delta_i&=\delta_{i+1}\delta_i\delta_{i+1}&&\text{for }1\leq{i}\leq\tfrac{p-5}{2},\\
\delta_0\delta_1\delta_0\delta_1&=\delta_1\delta_0\delta_1\delta_0&&\text{if }p>3,\\
\delta_{\nicefrac{(p-3)}{2}}\delta_{\nicefrac{(p-1)}{2}}\delta_{\nicefrac{(p-3)}{2}}\delta_{\nicefrac{(p-1)}{2}}&=\delta_{\nicefrac{(p-1)}{2}}\delta_{\nicefrac{(p-3)}{2}}\delta_{\nicefrac{(p-1)}{2}}\delta_{\nicefrac{(p-3)}{2}}&&\text{if }p>3.\end{align*}
For coprime odd integers $p,q\geq3$, we define a \textit{level} $q$ \textit{action} of $\mathfrak{W}_p$ on $\mathbb{Z}$ \cite{F3}: 
\begin{align*}\delta_0x&=\begin{cases}x-2q&\text{if }x\equiv q\:(\text{mod }p),\\
x+2q&\text{if }x\equiv-q\:(\text{mod }p),\\
x&\text{otherwise;}\end{cases}\\ 
\delta_ix&=\begin{cases}x-q&\text{if }x\equiv(i+1)q,-iq\:(\text{mod }p),\\
x+q&\text{if }x\equiv iq,-(i+1)q\:(\text{mod }p),\\
x&\text{otherwise,}\end{cases}&&\text{for }1\leq i\leq\tfrac{p-1}{2}.\end{align*}
For the rest of this paper, we will assume that $p$ and $q$ are coprime odd integers no less than 3. 
\begin{lem}
The above defines a group action of $\mathfrak{W}_p$ on $\mathbb{Z}$, and this can be extended to an action on $\mathcal{P}_2$.
\end{lem}
\begin{proof}
We must have $p>3$ for the fourth and fifth relations of the generators of $\mathfrak{W}_p$ to hold for the group action. 
The first relation $\delta_i^2=1$ is clear for all $i$. Moreover, the generators $\delta_i,\delta_j$ commute when $0\leq i<j-1\leq \nicefrac{(p-3)}{2}$ because they act on distinct congruence classes of integers modulo $p$. For the third relation, when $1\leq i\leq\nicefrac{(p-5)}{2}$ and $x\in\mathbb{Z}$ we have 
\begin{align*}
\delta_i\delta_{i+1}\delta_ix&=\left\{\begin{array}{l|l}\delta_i(x-2q)&x\equiv-iq\\
\delta_i(x+2q)&x\equiv iq\\
\delta_i(x-q)&x\equiv(i+1)q,(i+2)q\\
\delta_i(x+q)&x\equiv-(i+1)q,-(i+2)q\\
\delta_ix&\text{otherwise}\end{array}\right\}&&=\left\{\begin{array}{l|l}x-2q&x\equiv(i+2)q,-iq\\
x+2q&x\equiv iq,-(i+2)q\\
x&\text{otherwise}\end{array}\right.\\
&=\left\{\begin{array}{l|l}\delta_{i+1}(x-2q)&x\equiv(i+2)q\\
\delta_{i+1}(x+2q)&x\equiv-(i+2)q\\
\delta_{i+1}(x-q)&x\equiv-iq,-(i+1)q\\
\delta_{i+1}(x+q)&x\equiv iq,(i+1)q\\
\delta_{i+1}x&\text{otherwise}\end{array}\right\}&&=\delta_{i+1}\delta_i\delta_{i+1}x.\end{align*}
For the fourth and fifth relations, assuming $p>3$, 
\begin{align*}(\delta_0\delta_1)^2x&=\left\{\begin{array}{l|l}\delta_0\delta_1(x-3q)&x\equiv2q\\
\delta_0\delta_1(x+3q)&x\equiv-2q\\
\delta_0\delta_1(x-q)&x\equiv-q\\
\delta_0\delta_1(x+q)&x\equiv q\\
\delta_0\delta_1x&\text{otherwise}\end{array}\right\}&&=\left\{\begin{array}{l|l}x-4q&x\equiv2q\\
x+4q&x\equiv-2q\\
x-2q&x\equiv q\\
x+2q&x\equiv-q\\
x&\text{otherwise}\end{array}\right.\\
&=\left\{\begin{array}{l|l}\delta_1\delta_0(x-3q)&x\equiv q\\
\delta_1\delta_0(x+3q)&x\equiv-q\\
\delta_1\delta_0(x-q)&x\equiv2q\\
\delta_1\delta_0(x+q)&x\equiv-2q\\
\delta_1\delta_0x&\text{otherwise}\end{array}\right\}&&=(\delta_1\delta_0)^2x,\\
(\delta_{\nicefrac{(p-3)}{2}}\delta_{\nicefrac{(p-1)}{2}})^2x&=\left\{\begin{array}{l|l}\delta_{\nicefrac{(p-3)}{2}}\delta_{\nicefrac{(p-1)}{2}}(x-2q)&x\equiv\tfrac{(p+1)q}{2}\\
\delta_{\nicefrac{(p-3)}{2}}\delta_{\nicefrac{(p-1)}{2}}(x+2q)&x\equiv\tfrac{(p-1)q}{2}\\
\delta_{\nicefrac{(p-3)}{2}}\delta_{\nicefrac{(p-1)}{2}}(x-q)&x\equiv\tfrac{(p+3)q}{2}\\
\delta_{\nicefrac{(p-3)}{2}}\delta_{\nicefrac{(p-1)}{2}}(x+q)&x\equiv\tfrac{(p-3)q}{2}\\
\delta_{\nicefrac{(p-3)}{2}}\delta_{\nicefrac{(p-1}{2}}x&\text{otherwise}\end{array}\right\}&&=\left\{\begin{array}{l|l}x-3q&x\equiv\tfrac{(p+3)q}{2}\\
x+3q&x\equiv\tfrac{(p-3)q}{2}\\
x-q&x\equiv\tfrac{(p+1)q}{2}\\
x+q&x\equiv\tfrac{(p-1)q}{2}\\
x&\text{otherwise}\end{array}\right.\\
&=\left\{\begin{array}{l|l}\delta_{\nicefrac{(p-1)}{2}}\delta_{\nicefrac{(p-3)}{2}}(x-2q)&x\equiv\tfrac{(p+3)q}{2}\\
\delta_{\nicefrac{(p-1)}{2}}\delta_{\nicefrac{(p-3)}{2}}(x+2q)&x\equiv\tfrac{(p-3)q}{2}\\
\delta_{\nicefrac{(p-1)}{2}}\delta_{\nicefrac{(p-3)}{2}}(x-q)&x\equiv\tfrac{(p-1)q}{2}\\
\delta_{\nicefrac{(p-1)}{2}}\delta_{\nicefrac{(p-3)}{2}}(x+q)&x\equiv\tfrac{(p+1)q}{2}\\
\delta_{\nicefrac{(p-1)}{2}}\delta_{\nicefrac{(p-3)}{2}}x&\text{otherwise}\end{array}\right\}&&=(\delta_{\nicefrac{(p-1)}{2}}\delta_{\nicefrac{(p-3)}{2}})^2x.
\end{align*}

If $X$ is a subset of $\mathbb{Z}\backslash\{0\}$ that is bounded above, and its complement in $\mathbb{Z}$ is bounded below, then it is easy to see that the same is true for $aX:=\{ax|x\in X\}$, for any $a\in\mathfrak{W}_p$. Moreover, when $x\in X\Leftrightarrow-x\not\in X$, for all $x\in\mathbb{Z}\backslash\{0\}$, then the set $aX$ also satisfies this rule. 
Hence, this action can be extended to an action on bar partitions $\lambda$ by defining $\delta_i\lambda$ to be the bar partition with $\mathcal{A}(\delta_i\lambda)=\delta_i\mathcal{A}(\lambda)$.
\end{proof}

\begin{exam}
In our first example we illustrated the bead configuration 
$$\mathcal{A}((9,8,7,5,3))=\{9,8,7,5,3,-1,-2,-4,-6,-10,-11,-12,\dots\}.$$ 
We obtain the bead configuration of the bar partition $(13,6,5,2)=(\delta_0\delta_2)(9,8,7,5,3)$ by first subtracting $q:=3$ from each element in $\mathcal{A}((9,8,7,5,3))$ that is congruent to 4 (mod 5), and adding 3 to each element congruent to 1 (mod 5), then subtracting 6 from each element congruent to 3 (mod 5), and adding 6 to each element congruent to 2 (mod 5). We thus obtain the set 
$$\{13,6,5,2,-1,-3,-4,-7,-8,-9,-10,-11,-12,-14,-15,\dots\}=\mathcal{A}((13,6,5,2))$$ 
illustrated on the abacus below: 
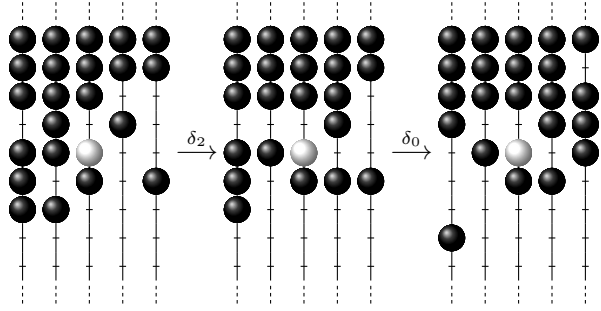
\begin{center}\abacus(vvvvv,bbbbb,bbbbb,bbbnn,nbnbn,bbonn,bnbnb,bbnnn,nnnnn,nnnnn,vvvvv)\:\raisebox{2cm}{$\overset{\delta_2}{\longrightarrow}$}\:\abacus(vvvvv,bbbbb,bbbbb,bbbbn,nnnbn,bbonn,bnbbb,bnnnn,nnnnn,nnnnn,vvvvv)\:\raisebox{2cm}{$\overset{\delta_0}{\longrightarrow}$}\:\abacus(vvvvv,bbbbb,bbbbn,bbbbb,bnnbb,nbonb,nnbbn,nnnnn,bnnnn,nnnnn,vvvvv)\end{center}
Notice that both bar partitions have the same $q$-bar-core; this will always be the case, as the level $q$ action defined on the generators of $\mathfrak{W}_p$ always corresponds to adding or removing $q$-bars. On the 3-runner abacus, the action of $\delta_0\delta_2\in\mathfrak{W}_5$ has the following effect on the bead configuration of $(9,8,7,5,3)$. 
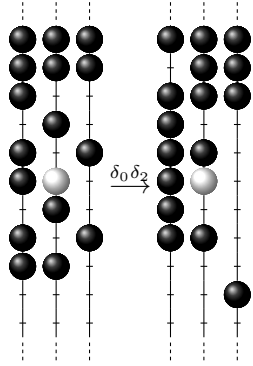
\begin{center}\abacus(vvv,bbb,bbb,bnn,nbn,bnb,bon,nbn,bnb,bbn,nnn,nnn,vvv)\:\raisebox{2.3cm}{$\overset{\delta_0\delta_2}{\longrightarrow}$}\:\abacus(vvv,bbb,nbb,bbb,bnn,bbn,bon,bnn,bbn,nnn,nnb,nnn,vvv)\end{center}
\end{exam}

We now give some invariants of the level $q$ action of $\mathfrak{W}_p$ which we will later use to give an explicit criterion for when two bar partitions lie in the same orbit under the level $q$ action, which we refer to as a \textbf{level} $q$ \textbf{orbit}.

\begin{lem}\label{lem4} 
Suppose $\lambda\in\mathcal{P}_2$ and $a\in\mathfrak{W}_p$, and define $a\lambda$ using the level $q$ action. 
\begin{enumerate}
\item $\overline{(a\lambda)}_q=\overline{\lambda}_q$; 
\item $\mathcal{Q}_p(a\lambda)$ is the same as $\mathcal{Q}_p(\lambda)$ with the components reordered; 
\item $\overline{\text{wt}}_p(a\lambda)=\overline{\text{wt}}_p(\lambda)$; 
\item $\overline{(a\lambda)}_p=a(\overline{\lambda}_p)$.
\end{enumerate}
\end{lem}
\begin{proof}
The relations occurring in all four parts are transitive, so we need only prove them in the case where $a$ is simply a generator $\delta_i$ of $\mathfrak{W}_p$. 

\vspace{3mm}
\noindent
(1) An element $x\in\mathcal{A}(\lambda)$ is fixed by the level $q$ action of $\delta_0$ if and only if 
$$x\not\equiv q,-q\:(\text{mod }p);\:x\equiv q\:(\text{mod }p),\:x-2q\in\mathcal{A}(\lambda);\text{ or }x\equiv-q\:(\text{mod }p),\:x+2q\in\mathcal{A}(\lambda).$$ 
If $x\in\mathcal{A}(\lambda)$, $x\equiv q$ (mod $p$) and $x-2q\not\in\mathcal{A}(\lambda)$, then $\delta_0$ sends $x$ to $x-2q$, and sends $2q-x\in\mathcal{A}(\lambda)$ to $-x$. If $x\in\mathcal{A}(\lambda)$, $x\equiv-q$ (mod $p$) and $x+2q\not\in\mathcal{A}(\lambda)$, then $\delta_0$ sends $x$ to $x+2q$, and sends $-x-2q\in\mathcal{A}(\lambda)$ to $-x$. Thus the action of $\delta_0$ on $\mathcal{A}(\lambda)$ corresponds to removing $2q$-bars from or adding $2q$-bars to $\lambda$. 

For $i\in\{1,\dots,\nicefrac{(p-1)}{2}\}$, the level $q$ action of $\delta_i$ fixes $x\in\mathcal{A}(\lambda)$ if and only if 
\begin{align*}x\not\equiv(i+1)q,-iq,iq,-(i+1)q\:(\text{mod }p);\:x\equiv(i+1)q,-iq\:(\text{mod }p),\:x-q\in\mathcal{A}(\lambda);\\
\text{or }x\equiv iq,-(i+1)q\:(\text{mod }p),\:x+q\in\mathcal{A}(\lambda).\end{align*} 
If $x\in\mathcal{A}(\lambda)$, $x\equiv(i+1)q$ or $-iq$ (mod $m$) and $x-q\not\in\mathcal{A}(\lambda)$, then $\delta_i$ sends $x$ to $x-q$, and sends $q-x\in\mathcal{A}(\lambda)$ to $-x$. If $x\in\mathcal{A}(\lambda)$, $x\equiv iq$ or $-(i+1)q$ (mod $p$) and $x+q\not\in\mathcal{A}(\lambda)$, then $\delta_i$ sends $x$ to $x+q$, and sends $-x-q\in\mathcal{A}(\lambda)$ to $-x$. Thus the action of $\delta_i$ on $\mathcal{A}(\lambda)$ corresponds to removing $q$-bars from or adding $q$-bars to $\lambda$. 

Moreover, since there are only finitely many $x\in\mathcal{A}(\lambda)$ such that at least one of $x-2q$, $x+2q$, $x-q$ or $x+q$ is not in $\mathcal{A}(\lambda)$, we obtain $\mathcal{A}(\delta_i\lambda):=\delta_i\mathcal{A}(\lambda)$ from $\mathcal{A}(\lambda)$ via a finite number of changes. Hence $\overline{(\delta_i\lambda)}_q=\overline{\lambda}_q$. 

\vspace{3mm}
\noindent
(2) When $x\not\equiv\pm q$ (mod $p$), we have $x\in\mathcal{A}(\delta_0\lambda)\Leftrightarrow x\in\mathcal{A}(\lambda)$, so $(\delta_0\lambda)^{(j)}=\lambda^{(j)}$ for all $j\not\equiv\pm q$ (mod $p$). If instead $x\equiv q$ (mod $p$), then $x\in\mathcal{A}(\delta_0\lambda)\Leftrightarrow x+2q\in\mathcal{A}(\lambda)$, or if $x\equiv-q$ (mod $p$), then $x\in\mathcal{A}(\delta_0\lambda)\Leftrightarrow x-2q\in\mathcal{A}(\lambda)$; hence $(\delta_0\lambda)^{(q)}=\lambda^{(-q)}=(\lambda^{(q)})'$ and $(\delta_0\lambda)^{(-q)}=\lambda^{(q)}=(\lambda^{(-q)})'$. 

If $i\in\{1,\dots,\nicefrac{(p-1)}{2}\}$ and $x\not\equiv(i+1)q,-iq,iq,-(i+1)q$ (mod $p$), then $x\in\mathcal{A}(\delta_i\lambda)\Leftrightarrow x\in\mathcal{A}(\lambda)$, so $(\delta_i\lambda)^{(j)}=\lambda^{(j)}$ for all $j\not\equiv(i+1)q,-iq,iq,-(i+1)q$ (mod $p$). If instead $x\equiv(i+1)q$ or $-iq$ (mod $p$), then $x\in\mathcal{A}(\delta_i\lambda)\Leftrightarrow x+q\in\mathcal{A}(\lambda)$, and if $x\equiv iq$ or $-(i+1)q$ (mod $p$), then $x\in\mathcal{A}(\delta_i\lambda)\Leftrightarrow x-q\in\mathcal{A}(\lambda)$; hence if $j\equiv(i+1)q$ or $-iq$ (mod $p$), then $(\delta_i\lambda)^{(j)}=\lambda^{(j-q)}$ and $(\delta_i\lambda)^{(-j)}=\lambda^{(j+q)}$. 

\vspace{3mm}
\noindent
(3) This follows from (2) and Lemma \ref{lem3}(1) (taking $c=1$): the bead configuration for $\delta_i\lambda$ on the $p$-runner abacus is the same as that of $\lambda$ but with the runners reordered, so the $p$-bar-weights of the two bar partitions are equal. 

\vspace{3mm}
\noindent
(4) We need to show that $\Delta_{j\text{ mod }p}(\delta_i\lambda)=\delta_i(\Delta_{k\text{ mod }p}\lambda)$, for each $i\in\{0,\dots,\nicefrac{(p-1)}{2}\}$ and $j\equiv\delta_ik$ (mod $p$). We suppose the contrary. \newline
When $i=0$, we may assume that $j\equiv\pm q\equiv-k$ (mod $p$), as otherwise $j=k$ and 
$$\Delta_j(\delta_0\lambda)=\Delta_j\lambda=\Delta_k\lambda=\delta_0(\Delta_k\lambda).$$ 
But $$j\equiv\pm q\equiv-k\Rightarrow\Delta_j(\delta_0\lambda)=\Delta_k\lambda\pm2q=\delta_0(\Delta_k\lambda),$$ 
so we must have $i>0$, and we may also assume that $j\equiv\pm iq$ or $\pm(i+1)q$, and $j\equiv k\pm q$ (mod $p$). But then 
$$\Delta_j(\delta_i\lambda)=\Delta_{\delta_i(j)}\lambda=\delta_i(\Delta_j\lambda),$$ 
so in fact, the $p$-set of $\delta_i\lambda$ is equal to the image of the $p$-set of $\lambda$ under the action of $\delta_i$. Hence $\overline{(\delta_i\lambda)}_p=\delta_i(\overline{\lambda}_p)$, for all $i\in\{0,\dots,\nicefrac{(p-1)}{2}\}$. 
\end{proof}

Next we will give a criterion for when two bar partitions lie in the same level $q$ orbit; to this end, we will first establish a condition for two $p$-bar-cores to lie in the same level $q$ orbit.

\begin{prop}\label{prop7} 
Suppose $\lambda,\mu\in\overline{C}_p$, and that the multisets 
$$[\Delta_{i\text{ mod }p}\lambda\:(\text{mod }q)|i\equiv0,\dots,p-1],\:[\Delta_{i\text{ mod }p}\mu\:(\text{mod }q)|i\equiv0,\dots,p-1]$$
are equal. Then $\overline{\lambda}_q=\overline{\mu}_q$, and $\lambda$ and $\mu$ lie in the same level $q$ orbit \textup{(}of $p$-bar-cores\textup{)}. 
\end{prop}
\begin{proof}
The fact that $\lambda$ and $\mu$ have the same $q$-bar-core is established by Fayers' \cite[Proposition 4.1]{F3}, and it follows from the definition of the level $q$ action of $\mathfrak{W}_p$ on the set of $p$-bar-cores that this action preserves the $q$-bar-core of a bar partition as by definition $\delta_i$ does not change the multiset of residues modulo $q$ of the elements of the $p$-set. Therefore each orbit of the level $q$ action on $\overline{C}_p$ can contain at most one $q$-bar-core. 

In the same paper \cite{F3}, Fayers proves the following result: 
\begin{center}
Suppose $\mathcal{O}$ is a level $q$ orbit. Let $\nu$ be an element of $\mathcal{O}$ for which the sum $\sum^{p-1}_{i=0}(\Delta_{i\text{ mod }p}\nu-\nicefrac{p}{2})^2$ is minimised. Then $\nu$ is a $q$-bar-core. 
\end{center}
This $\nu$ is uniquely defined as each level $q$ orbit contains no more than one $q$-bar-core. Thus letting $\nu$ be the $q$-bar-core of both $\lambda$ and $\mu$, it must be contained in both the level $q$ orbit containing $\lambda$ and the level $q$ orbit containing $\mu$; so these orbits coincide. 
\end{proof}

For the more general result, we define the $q$\textbf{-weighted} $p$\textbf{-quotient} of $\lambda\in\mathcal{P}_2$ with $p$-set $\{\Delta_0\lambda,\dots,\Delta_{p-1}\lambda\}$ and $p$-quotient $\mathcal{Q}_p(\lambda)=(\lambda^{(0)},\dots,\lambda^{(p-1)})$ to be the multiset 
$$\mathcal{Q}^q_p(\lambda):=[(\Delta_i\lambda\:(\text{mod }q),\lambda^{(i)})|i\equiv0,1,\dots,p-1].$$ 

\begin{exam}
We have already seen that the bar partition $(9,8,7,5,3)$ has 5-set $\{0,-4,-3,8,9\}$ and $5$-quotient $((1),(1),(3),(1^3),(1))$, so $\mathcal{Q}^3_5((9,8,7,5,3))$ is the multiset 
$$[(0\:(\text{mod }3),(1)),(2\:(\text{mod }3),(1)),(0\:(\text{mod }3),(3)),(2\:(\text{mod }3),(1^3)),(0\:(\text{mod }3),(1))].$$
\end{exam}

\begin{prop}\label{prop8} 
Suppose $\lambda,\mu\in\mathcal{P}_2$. Then $\lambda$ and $\mu$ lie in the same level $q$ orbit of $\mathfrak{W}_p$ if and only if they have the same $q$-weighted $p$-quotient. 
\end{prop}
\begin{proof}
Firstly suppose that $\lambda$ and $\mu$ lie in the same level $q$ orbit; we may assume that $\mu=\delta_i\lambda$ for some $i\in\{0,\dots,\nicefrac{(p-1)}{2}\}$. Then we get $\mathcal{Q}^q_p(\lambda)=\mathcal{Q}^q_p(\mu)$ from the proof of Lemma \ref{lem4}: when $\mu=\delta_0\lambda$, we have 
$$(\Delta_j\mu,\mu^{(j)})=\begin{cases}(\Delta_{-j}\lambda+2q,\lambda^{(-j)})&\text{for }j\equiv q\text{ (mod }p),\\
(\Delta_{-j}\lambda-2q,\lambda^{(-j)})&\text{for }j\equiv-q\text{ (mod }p),\\
(\Delta_j\lambda,\lambda^{(j)})&\text{otherwise},\end{cases}$$ 
and when $\mu=\delta_i\lambda$ for some $i\in\{1,\dots,\nicefrac{(p-1)}{2}\}$, we have 
$$(\Delta_j\mu,\mu^{(j)})=\begin{cases}(\Delta_{j-q}\lambda+q,\lambda^{(j-q)})&\text{for }j\equiv(i+1)q\text{ or }-iq\text{ (mod }p),\\
(\Delta_{j+q}\lambda-q,\lambda^{(j+q)})&\text{for }j\equiv iq\text{ or }-(i+1)q\text{ (mod }p),\\
(\Delta_j\lambda,\lambda^{(j)})&\text{otherwise}.\end{cases}$$ 

For the other direction, suppose that $\lambda$ and $\mu$ share the $q$-weighted $p$-quotient $\mathcal{Q}^q_p(\lambda)=\mathcal{Q}^q_p(\mu)$. By the definition of the $p$-set, and since all components of the $p$-quotient of a $p$-bar-core are equal to the empty bar partition, the $p$-bar-cores of $\lambda$ and $\mu$ must have the same $q$-weighted $p$-quotient $\mathcal{Q}^q_p(\overline{\lambda}_p)=\mathcal{Q}^q_p(\overline{\mu}_p)$. Thus, by Proposition \ref{prop7} we may find $a,b\in\mathfrak{W}_p$ such that $a(\overline{\lambda}_p)=b(\overline{\mu}_p)=\sigma$, where $\sigma$ is the $q$-bar-core of both $\overline{\lambda}_p$ and $\overline{\mu}_p$. Then by Lemma \ref{lem4}(4) we have $\overline{(a\lambda)}_p=\overline{(b\mu)}_p=\sigma$, so using Lemma \ref{lem4}(1) we see that $\sigma$ is the $p$-bar-core and the $q$-bar-core of both $a\lambda$ and $b\mu$; in particular, $a\lambda$ and $b\mu$ have the same $p$-set. Moreover, by our assumption and the proof of the only `only if' part of the proposition above both $a\lambda$ and $b\mu$ have $q$-weighted $p$-quotient $\mathcal{Q}^q_p(\lambda)$. 

From the proof of \cite[Proposition 4.1]{F3} and the fact that $\sigma\in\overline{C}_p\cap\overline{C}_q$ it follows that for each $k\in\{0,\dots, q-1\}$, the elements $\Delta_j\sigma$ in the $p$-set (of $\sigma$, $a\lambda$ and $b\mu$) that are congruent to $k$ modulo $q$ form an arithmetic progression with common difference $q$. By the first paragraph of this proof we can therefore apply the level $q$ action to $a\lambda$ and arbitrarily reorder the elements $(a\lambda)^{(j)}\in\mathcal{Q}_p(a\lambda)$ such that $j\equiv k$ (mod $q$), for each $k$, without affecting the $q$-weighted $p$-quotient. 

Thus we can apply elements of $\mathfrak{W}_p$ to transform $a\lambda$ to $b\mu$, as $\mathcal{Q}^q_p(a\lambda)=\mathcal{Q}^q_p(b\mu)$ if and only if the multisets $[(a\lambda)^{(j)}|\Delta_j\sigma\equiv k\text{ (mod }q)]\text{ and }[(b\mu)^{(j)}|\Delta_j\sigma\equiv k\text{ (mod }q)]$ are equal for each $k\in\{0,\dots,q-1\}$, and it follows that $\lambda$ and $\mu$ lie in the same level $q$ orbit. 
\end{proof}

\section{Generalised bar-cores}
\noindent
Now that we have covered all of the necessary definitions and basic results relating to the action of $\mathfrak{W}_p$, we arrive at the first of our main results. The following proposition is a generalisation of a theorem by Olsson \cite[Theorem 4]{O2} which states that the $q$-bar-core of a $p$-bar-core is again a $p$-bar-core, or in the notation used above, 
$$\overline{\text{wt}}_p(\lambda)=0\Rightarrow\overline{\text{wt}}_p(\overline{\lambda}_q)=0.$$ 

\begin{prop}\label{prop1}
For all bar partitions $\lambda$, 
$$\overline{\text{wt}}_p(\overline{\lambda}_q)\leq\overline{\text{wt}}_p(\lambda).$$
\end{prop}
\begin{proof}
We use induction on $\overline{\text{wt}}_q(\lambda)$, with the trivial case being that $\lambda$ is a $q$-bar-core. Assuming that this is not the case, we may find a removable $q$-bar: $y\in\lambda$ such that $y-q\not\in\mathcal{A}(\lambda)$. We will describe how to remove $q$-bars from $\lambda$ to obtain a new partition with the same $q$-bar-core as $\lambda$, with $q$-bar-weight strictly less than $\overline{\text{wt}}_q(\lambda)$, and with $p$-bar-weight no more than $\overline{\text{wt}}_p(\lambda)$.

Let $y$ be any part of $\lambda$ such that $y-q\not\in\mathcal{A}(\lambda)$. For any $x\in\mathcal{A}(\lambda)$ congruent to $y$ modulo $p$ such that $x-q\not\in\mathcal{A}(\lambda)$, replace $x$ with $x-q$, then replace $q-x\in\mathcal{A}(\lambda)$ with $-x$. We keep repeating this process until there are no more such $x$ (the process will terminate because $\lambda$ has finitely many removable $q$-bars), then we name our new bar partition $\nu$. Since each action corresponds to removing a $q$-bar from $\lambda$, and since we have removed at least one $q$-bar (replacing $y$ with $y-q$, and $q-y$ with $-y$, in $\mathcal{A}(\lambda)$), we have 
$$\overline{\nu}_q=\overline{\lambda}_q\text{ and }\overline{\text{wt}}_q(\nu)<\overline{\text{wt}}_q(\lambda).$$ 

We remarked earlier that the $p$-bar-weight of a bar partition $\lambda$ is equal to half the number of pairs $(x,a)\in\mathcal{A}(\lambda)\times\mathbb{N}$ such that $x-ap\not\in\mathcal{A}(\lambda)$. We will call such a pair $(x,a)$ a $p$\textbf{-bar-weight pair for} $\lambda$. It follows from our construction of $\nu$ that for any $x\not\equiv y,y-q,q-y,-y$ (mod $p$) and $a\in\mathbb{N}$, $(x,a)$ is a $p$-bar-weight pair for $\nu$ if and only if it is a $p$-bar-weight pair for $\lambda$. We will consider the remaining possibilities for the residue of $y$ modulo $p$ and show that in each case $\nu$ has no more $p$-bar-weight pairs than $\lambda$, hence $\overline{\text{wt}}_p(\nu)\leq\overline{\text{wt}}_p(\lambda)$. 

First suppose that $y\equiv q$ (mod $p$), so that we obtain $\mathcal{A}(\nu)$ by repeatedly replacing each $x\in\mathcal{A}(\lambda)$ such that $x\equiv q$ (mod $p$) and $q-x\in\mathcal{A}(\lambda)$ with $x-q$, then replacing $q-x$ with $-x$, until there are no more such $x$. Since in this case $y-q\equiv0\equiv q-y$ (mod $p$), we may compare the $p$-bar-weights of $\lambda$ and $\nu$ by counting how many of the three pairs $(x,a)$, $(x-q,a)$, $(x-2q,a)$ are $p$-bar-weight pairs for each of the two bar partitions when $x\equiv q$ (mod $p$) and $a\in\mathbb{N}$. We will do this by considering each of the four possibilities for the size of $X:=\mathcal{A}(\lambda)\cap\{x,x-q,x-2q\}$. 
\begin{flushleft}
$|X|=3$: If $x,x-q,x-2q\in\mathcal{A}(\lambda)$, then $x,x-q,x-2q\in\mathcal{A}(\nu)$, so the number of $p$-bar-weight pairs for $\lambda$, and for $\nu$, amongst the three pairs $(x,a)$, $(x-q,a)$, and $(x-2q,a)$ is equal to $3-|\mathcal{A}(\lambda)\cap\{x-ap,x-ap-q,x-ap-2q\}|$. 
\end{flushleft}
\begin{flushleft}
$|X|=2$: We have $\mathcal{A}(\nu)\cap\{x,x-q,x-2q\}=\{x-q,x-2q\}$, so clearly $(x,a)$ is not a $p$-bar-weight pair for $\nu$. If only one of the three pairs is a $p$-bar-weight pair for $\nu$, it must be $(x-q,a)$ as necessarily $\mathcal{A}(\nu)\cap\{x-ap,x-ap-q,x-ap-2q\}=\{x-ap-2q\}$. If $(x-q,a)$ and $(x-2q,a)$ are both $p$-bar-weight pairs for $\nu$, then we must have $\mathcal{A}(\lambda)\cap\{x-ap,x-ap-q,x-ap-2q\}=\varnothing$, so $\lambda$ also has two $p$-bar-weight pairs out of the three.
\end{flushleft}
\begin{flushleft}
$|X|=1$: We have $\mathcal{A}(\nu)\cap\{x,x-q,x-2q\}=\{x-2q\}$, so neither of $(x,a)$, $(x-q,a)$ can be $p$-bar-weight pairs for $\nu$. If $\mathcal{A}(\lambda)\cap\{x-ap,x-ap-q,x-ap-2q\}=\varnothing$, then exactly one of $(x,a)$, $(x-q,a)$, $(x-2q,a)$ is a $p$-bar-weight pair for $\lambda$. If $|\mathcal{A}(\lambda)\cap\{x-ap,x-ap-q,x-ap-2q\}|\geq1$, then none of $(x,a)$, $(x-q,a)$, $(x-2q,a)$ can be $p$-bar-weight pairs for $\nu$ as $x-ap-2q\in\mathcal{A}(\nu)$. 
\end{flushleft}
\begin{flushleft}
$|X|=0$: Since $\mathcal{A}(\nu)\cap\{x,x-q,x-2q\}=\varnothing$, none of $(x,a)$, $(x-q,a)$, $(x-2q,a)$ are $p$-bar-weight pairs for $\nu$. 
\end{flushleft}

Next suppose that $y\equiv0$ (mod $p$), so that we obtain $\mathcal{A}(\nu)$ by repeatedly replacing each $x\in\mathcal{A}(\lambda)$ such that $p|x$ and $q-x\in\mathcal{A}(\lambda)$ with $x-q$, then replacing $q-x$ with $-x$, until there are no much such $x$. Since $y\equiv-y$ (mod $p$), we can apply the same argument as above, when $y\equiv q$ (mod $p$), and conclude that $\nu$ has no more $p$-bar-weight pairs than $\lambda$ amongst $(x+q,a)$, $(x,a)$, $(x-q,a)$, and thus $\overline{\text{wt}}_p(\nu)\leq\overline{\text{wt}}_p(\lambda)$. 

Finally, suppose that $y\not\equiv q,0$ (mod $p$), so that $y\not\equiv-y$ and $y-q\not\equiv q-y$. In this case, we need only consider replacing all pairs $x,q-x\in\mathcal{A}(\lambda)$ such that $x\equiv y$ (mod $p$) with $x-q,-x$, so we are in a simpler situation; $\overline{\text{wt}}_p(\nu)\leq\overline{\text{wt}}_p(\lambda)$ since for any $x\equiv y$ (mod $p$) and $a\in\mathbb{N}$, $\nu$ has no more $p$-bar-weight pairs than $\lambda$ amongst $(x,a)$ and $(x-q,a)$.

Hence $\nu$ has no more $p$-bar-weight pairs than $\lambda$, and therefore has $p$-bar-weight no more than the $p$-bar-weight of $\lambda$. The result follows by induction.
\end{proof}

From this purely combinatorial result we obtain an interesting algebraic corollary. 

\begin{cor}For any $\mu\in\mathcal{P}_2$, if $w$ is the weight of the $p$-block containing $[\mu]$, a spin representation of the symmetric group $\mathfrak{S}_r$ ($r\in\mathbb{N}$), and $[\lambda]$ is a spin representation of $\mathfrak{S}_{r+iq}$ corresponding to $\lambda\in\mathcal{P}_2$ obtained by adding $q$-bars to $\mu$, for any $i\in\mathbb{N}$, then $[\lambda]$ belongs to a $p$-block of weight $\geq w$. \newline
In particular, if $[\mu]$ belongs to a $p$-block of weight $w>0$, then $[\lambda]$ belongs to a block of positive weight. 
\end{cor}

Next we will consider the set $\overline{C}_{p,q}$ containing all bar partitions $\lambda$ which satisfy 
$$\overline{\text{wt}}_p(\lambda)=\overline{\text{wt}}_p(\overline{\lambda}_q).$$

\begin{prop}\label{prop2}
For all $\lambda\in\mathcal{P}_2$, the equality $\overline{\text{wt}}_p(\overline{\lambda}_q)=\overline{\text{wt}}_p(\lambda)$ holds if and only if there do not exist integers $a,b,c$ such that: 
\begin{align*}a\equiv b\:(\text{mod }p);\\
a\equiv c\:(\text{mod }q);\\
a,b+c-a\in\mathcal{A}(\lambda);\\
b,c\not\in\mathcal{A}(\lambda).\end{align*}
\end{prop}
\begin{proof}
Say that $(a,b,c)$ is a bad triple for $\lambda$ if $a,b,c$ satisfy the conditions above. When $(a,b,c)$ is bad, either $a>c$ or $b+c-a>b$; either way, since $a\equiv c$ (mod $q$) we find that $\lambda$ has a removable $q$-bar and is thus not a $q$-bar-core. Hence the proposition is true when $\lambda\in\overline{C}_q$. 

Now we assume $\lambda$ is not a $q$-bar-core, and choose $y\in\lambda$ such that $y-q\not\in\mathcal{A}(\lambda)$. We define a new bar partition $\nu$ as in the proof of Proposition \ref{prop1}: by repeatedly replacing pairs $x,q-x\in\mathcal{A}(\lambda)$ with $x-q$ and $-x$, respectively, when $x\equiv y$ (mod $p$). \newline
By induction it suffices to show that either: \newline
$\overline{\text{wt}}_p(\nu)=\overline{\text{wt}}_p(\lambda)$, and there is a bad triple for $\nu$ \textit{iff}. there is a bad triple for $\lambda$; or \newline
$\overline{\text{wt}}_p(\nu)<\overline{\text{wt}}_p(\lambda)$, and there is a bad triple for $\lambda$.

Suppose first that there are no pairs $x,q-x$ satisfying 
\begin{align}\label{ass}x,q-x\not\in\mathcal{A}(\lambda)\text{ and }x\equiv y\text{ (mod }p\text{)}.\end{align} 
We first assume that $y\equiv q$ (mod $p$), and let $x\equiv q$ (mod $p$). Then there are 8 different possibilities for the intersection of $\mathcal{A}(\lambda)$ and $\{x,x-q,x-2q\}$: 
\begin{align*}x,x-q,x-2q\in&\mathcal{A}(\lambda);\\
x,x-q\in&\mathcal{A}(\lambda)\not\ni x-2q;\\
x\in&\mathcal{A}(\lambda)\not\ni x-q,x-2q;\\
&\mathcal{A}(\lambda)\not\ni x,x-q,x-2q;\\
x,x-2q\in&\mathcal{A}(\lambda)\not\ni x-q;\\
x-q,x-2q\in&\mathcal{A}(\lambda)\not\ni x;\\
x-q\in&\mathcal{A}(\lambda)\not\ni x,x-2q;\\
x-2q\in&\mathcal{A}(\lambda)\not\ni x,x-q.\end{align*} 
However, the last four possibilities are all excluded by our assumption that there are no pairs $x,q-x$ satisfying (\ref{ass}), so we find that $\nu=\delta_0\lambda$. Therefore $\overline{\text{wt}}_p(\nu)=\overline{\text{wt}}_p(\lambda)$ by Lemma \ref{lem4}(3), and $(a,b,c)$ is bad for $\lambda$ exactly when $(\delta_0a,\delta_0b,\delta_0c)$ is bad for $\nu$, since $a\equiv b$ (mod $p$) $\Rightarrow \delta_0(b+c-a)=\delta_0b+\delta_0c-\delta_0a$.

If $y\equiv 0$ (mod $p$) we are in an identical situation to the above: $\nu=\delta_0\lambda$. 

When $y\not\equiv q,0$ (mod $p$), we have a similar situation: since there are no pairs $x,q-x$ satisfying (\ref{ass}), we have $\nu=\delta_i\lambda$, where $(i+1)q\equiv y$ (mod $p$). Hence $\overline{\text{wt}}_p(\nu)=\overline{\text{wt}}_p(\lambda)$, and $(a,b,c)$ is bad for $\lambda\Leftrightarrow (\delta_ia,\delta_ib,\delta_ic)$ is bad for $\nu$.

Finally, we assume that there is a pair $x,q-x$ satisfying (\ref{ass}), so that $(y,x,y-q)$ is a bad triple for $\lambda$. We argue that $\overline{\text{wt}}_p(\nu)<\overline{\text{wt}}_p(\lambda)$, as in the proof of Proposition \ref{prop1}: \newline
If $y\equiv q$ (mod $p$) and we let $z:=$ max$\{x,y\}$, $l:=\nicefrac{|x-y|}{p}$, then exactly one of $(z,l)$, $(z-q,l)$ is a $p$-bar-weight pair for $\lambda$ ($(z,l)$ if $x<y$, or $(z-q,l)$ if $x>y$). If $(z-2q,l)$ is a $p$-bar-weight pair for $\lambda$, then $(z-q,l)$ is a $p$-bar-weight pair for $\nu$ but neither of $(z,l)$, $(z-2q,l)$ are; and if $(z-2q,l)$ is not a $p$-bar-weight pair for $\lambda$, then none of $(z,l)$, $(z-q,l)$, $(z-2q,l)$ are $p$-bar-weight pairs for $\nu$. \newline
If instead we have $y\equiv 0$ (mod $p$), and again let $z:=$ max$\{x,y\}$ and $l:=\nicefrac{|x-y|}{p}$, then exactly one of $(z,l)$, $(z-q,l)$ is a $p$-bar-weight pair for $\lambda$. Now if $(z+q,l)$ is a $p$-bar-weight pair for $\lambda$, then $(z,l)$ is a $p$-bar-weight pair for $\nu$ but neither of $(z+q,l)$, $(z-q,l)$ are; and if $(z+q,l)$ is not a $p$-bar-weight pair for $\lambda$, then none of $(z+q,l)$, $(z,l)$, $(z-q,l)$ are $p$-bar-weight pairs for $\nu$. \newline
If $y\not\equiv q,0$ (mod $p$) then we define $z$ and $l$ as above so that exactly one of $(z,l),(z-q,l)$ is a $p$-bar-weight pair for $\lambda$ and neither is a $p$-bar-weight pair for $\nu$. \newline
Thus it follows from the proof of Proposition \ref{prop1} that there are less $p$-bar-weight pairs for $\nu$ then there are $p$-bar-weight pairs for $\lambda$. 
\end{proof}

\begin{cor}\label{cor2}
$\overline{C}_{p,q}=\overline{C}_{q,p}$. 
\end{cor}
\begin{proof}
The condition in Proposition \ref{prop2} is symmetric in $p$ and $q$.
\end{proof}

While the last result may seem surprising given the definition of $\overline{C}_{p,q}$, this symmetry is the motivation behind the study of this set. Furthermore, the next result shows that $\overline{C}_{p,q}$ is closed under the level $q$ action of $\mathfrak{W}_p$. 

\begin{prop}\label{prop3}
For any $\lambda\in\mathcal{P}_2$ and $a\in\mathfrak{W}_p$, if $\lambda\in\overline{C}_{p,q}$, then $a\lambda\in\overline{C}_{p,q}$. 
\end{prop}
\begin{proof}
Using Lemma \ref{lem4}($1,3$) and the fact that $\lambda\in\overline{C}_{p,q}$, we have 
$$\overline{\text{wt}}_p(\overline{(a\lambda)}_q)=\overline{\text{wt}}_p(\overline{\lambda}_q)=\overline{\text{wt}}_p(\lambda)=\overline{\text{wt}}_p(a\lambda).$$ 
\end{proof}

Interchanging $p$ and $q$ and appealing to Corollary \ref{cor2}, we see that $\overline{C}_{p,q}$ is also a union of orbits for the level $p$ action of $\mathfrak{W}_q$. The actions of $\mathfrak{W}_p$ and $\mathfrak{W}_q$ clearly commute because the action of $\mathfrak{W}_p$ on an integer does not change its residue modulo $q$, and the action of $\mathfrak{W}_q$ does not change its residue modulo $p$. Hence $\overline{C}_{p,q}$ is a union of orbits for the action of $\mathfrak{W}_p\times\mathfrak{W}_q$. We will look at these orbits in more detail, first by considering just the level $q$ action of $\mathfrak{W}_p$.

\begin{prop} \label{prop4}
Suppose $\lambda\in\mathcal{P}_2$, and let $\mathcal{O}$ be the orbit containing $\lambda$ under the level $q$ action of $\mathfrak{W}_p$. Then the following are equivalent: \begin{enumerate}
\item $\lambda\in\overline{C}_{p,q}$; 
\item $\mathcal{O}$ contains a $q$-bar-core; 
\item $\mathcal{O}$ contains $\overline{\lambda}_q$.
\end{enumerate}
\end{prop}
\begin{proof}
Since $\overline{C}_q\subset\overline{C}_{p,q}$, Proposition \ref{prop3} shows that if $\mathcal{O}$ contains a $q$-bar-core, then $\lambda\in\overline{C}_{p,q}$. Hence the second statement implies the first. Trivially the third statement implies the second, so it remains to show that the first implies the third. So suppose that $\lambda\in\overline{C}_{p,q}$, and we can assume that $\lambda$ is not a $q$-bar-core or the third statement is trivial. Thus we may find a pair $y,q-y\in\mathcal{A}(\lambda)$. By the proof of Proposition \ref{prop2}, there are no pairs $x,q-x\not\in\mathcal{A}(\lambda)$ with either $x$ or $q-x\equiv y$ (mod $p$), and if we take $i\in\{0,\dots,\nicefrac{(p-1)}{2}\}$ such that $iq\equiv y$ (mod $p$), then the bar partition $\nu=\delta_i\lambda$ satisfies $\overline{\nu}_q=\overline{\lambda}_q$ and $\overline{\text{wt}}_q(\nu)<\overline{\text{wt}}_q(\lambda)$. Since $\overline{\text{wt}}_p(\nu)=\overline{\text{wt}}_p(\delta_i\lambda)=\overline{\text{wt}}_p(\lambda)$, $\nu$ is also in $\overline{C}_{p,q}$, and by induction the orbit containing $\nu$ contains $\overline{\nu}_q$.
\end{proof}

\begin{cor}\label{cor3}
Let $\mathcal{O}$ be an orbit of $\mathfrak{W}_p\times\mathfrak{W}_q$ consisting of bar partitions in $\overline{C}_{p,q}$. Then $\mathcal{O}$ contains exactly one bar partition that is both a $p$-bar-core and a $q$-bar-core.
\end{cor}
\begin{proof}
Let $\lambda$ be a bar partition in $\mathcal{O}$. Then by Proposition \ref{prop4}, $\overline{\lambda}_q\in\mathcal{O}$, and by the same result with $p$ and $q$ interchanged, the bar partition $\nu=\overline{(\overline{\lambda}_q)}_p$ lies in $\mathcal{O}$. Obviously $\nu$ is a $p$-bar-core, and by Proposition \ref{prop1}, it is also a $q$-bar-core.

Now suppose that there is another bar partition in $\mathcal{O}$ that is both a $p$-bar-core and a $q$-bar-core. We can write this as $ba\nu$, with $a\in\mathfrak{W}_p$ and $b\in\mathfrak{W}_q$. Since $\overline{\text{wt}}_q(\delta_j\lambda)=\overline{\text{wt}}_q(\lambda)$ for any $j\in\{0,\dots,\nicefrac{(q-1)}{2}\}$ (by interchanging $p$ and $q$ in the proof of Proposition \ref{prop3}), it follows that 
$$\overline{\text{wt}}_q(a\nu)=\overline{\text{wt}}_q(ba\nu)=0,$$ 
hence it follows from $\overline{(\delta_i\lambda)}_q=\overline{\lambda}_q$ (for any $i\in\{0,\dots,\nicefrac{(p-1)}{2}\}$) that 
$$a\nu=\overline{a\nu}_q=\overline{\nu}_q=\nu.$$ 
Similarly $b\nu=\nu$, and thus $ba\nu=\nu$. 
\end{proof}

\begin{rema}
From Proposition \ref{prop4} and Corollary \ref{cor3}, we see that two bar partitions $\lambda,\mu\in\overline{C}_{p,q}$ lie in the same orbit of $\mathfrak{W}_p\times\mathfrak{W}_q$ if and only if the $p$-bar-cores of $(\overline{\lambda}_q)$ and $(\overline{\nu}_q)$ are equal. However, it does not seem to be easy to tell when two arbitrary bar partitions lie in the same orbit.
\end{rema}

\begin{lem}\label{lem1}
Suppose $\lambda\in\overline{C}_{p,q}$. Then $\lambda$ is a $pq$-bar-core, and $\overline{(\overline{\lambda}_q)}_p=\overline{(\overline{\lambda}_p)}_q$.
\end{lem}
\begin{proof}
If we can remove a $pq$-bar from $\lambda$ to obtain a new bar partition $\nu$, then we can also remove $q$ successive $p$-bars, or $p$ successive $q$-bars, to obtain $\nu$ from $\lambda$. Thus $\overline{\nu}_q=\overline{\lambda}_q$ and $\overline{\text{wt}}_p(\nu)\leq\overline{\text{wt}}_p(\lambda)-q$, so we have 
$$\overline{\text{wt}}_p(\overline{\lambda}_q)=\overline{\text{wt}}_p(\overline{\nu}_q)\leq\overline{\text{wt}}_p(\nu)<\overline{\text{wt}}_p(\lambda);$$ 
hence $\overline{\text{wt}}_{pq}(\lambda)>0\Rightarrow\lambda\not\in\overline{C}_{p,q}$. 

It follows from Proposition \ref{prop1} that $\overline{(\overline{\lambda}_q)}_p$ and $\overline{(\overline{\lambda}_p)}_q$ are both $p$-bar-cores and $q$-bar-cores, and by Proposition \ref{prop4} they both lie in the same orbit as $\lambda$ under the action of $\mathfrak{W}_p\times\mathfrak{W}_q$. Hence the result follows from Corollary \ref{cor3}.
\end{proof}

\section{The sum of a $p$-bar-core and a $q$-bar-core}
\noindent
In the present section we will give a constructive method for finding a bar partition in $\overline{C}_{p,q}$ with a given $p$-bar-core $\mu$ and $q$-bar-core $\sigma$. The resulting bar partition can be interpreted as the `sum' of $\mu$ and $\sigma$. 

\begin{prop}\label{prop5}
Suppose $\mu\in\overline{C}_p$ and $\sigma\in\overline{C}_q$, and that $\overline{\mu}_q=\overline{\sigma}_p$. Then there is a unique bar partition $\lambda\in\overline{C}_{p,q}$ with $\overline{\lambda}_p=\mu$ and $\overline{\lambda}_q=\sigma$. Moreover, 
$$|\lambda|=|\mu|+|\sigma|-|\overline{\sigma}_p|,$$ 
and $\lambda$ is the unique smallest bar partition with $p$-bar-core $\mu$ and $q$-bar-core $\sigma$.
\end{prop}
\begin{proof}
Let $\tau=\overline{\mu}_q$, and consider the action of $\mathfrak{W}_p\times\mathfrak{W}_q$ on $\mathcal{P}_2$. By Proposition \ref{prop4} we can find $a\in\mathfrak{W}_p$ and $b\in\mathfrak{W}_q$ such that $a\tau=\mu$ and $b\tau=\sigma$, and we let $\lambda=a\sigma$, so that $\lambda\in\overline{C}_{p,q}$ (as it lies in the same orbit as the $q$-bar-core $\sigma$). Then we have $\overline{\lambda}_q=\overline{\sigma}_q=\sigma$, and by the proof of Proposition \ref{prop3}, we have 
$$\overline{\lambda}_p=\overline{(ab\tau)}_p=\overline{(ba\tau)}_p=\overline{(b\mu)}_p=\overline{\mu}_p=\mu.$$ 
Moreover, we have 
\begin{align*}
|\lambda|&=|\overline{\lambda}_p|+p\cdot\overline{\text{wt}}_p(\lambda)\\
&=|\mu|+p\cdot\overline{\text{wt}}_p(a\sigma)\\
&=|\mu|+p\cdot\overline{\text{wt}}_p(\sigma)\\
&=|\mu|+|\sigma|-|\overline{\sigma}_p|.\end{align*} 

Now suppose $\nu$ is a bar partition distinct from $\lambda$ with $\overline{\nu}_p=\mu$ and $\overline{\nu}_q=\sigma$, and let $a,b$ be as above. Then we have $\overline{(a^{-1}\nu)}_q=\overline{\nu}_q=\sigma$, but $a^{-1}\nu\neq a^{-1}\lambda=\sigma$, so $|a^{-1}\nu|>|\sigma|$. Hence, again using the proof of Proposition \ref{prop3}, we have  
$$\overline{\text{wt}}_p(\nu)=\overline{\text{wt}}_p(a^{-1}\nu)=\tfrac{|a^{-1}\nu|-|\tau|}{p}>\tfrac{|\sigma|-|\tau|}{p}=\overline{\text{wt}}_p(\sigma)$$ 
which means that $\nu\not\in\overline{C}_{p,q}$. Furthermore, we see that $\overline{\text{wt}}_p(\nu)>\overline{\text{wt}}_p(\lambda)$, so $|\nu|>|\lambda|$. Hence $\lambda$ is the unique smallest bar partition with $p$-bar-core $\mu$ and $q$-bar-core $\sigma$.
\end{proof}

\begin{rema}
If $\mu,\sigma\in\mathcal{P}_2$ and $\lambda$ is the unique bar partition in $\overline{C}_{p,q}$ with $\overline{\lambda}_p=\mu$ and $\overline{\lambda}_q=\sigma$, then we have shown that $\lambda$ is the smallest bar partition with $p$-bar-core $\mu$ and $q$-bar-core $\sigma$ in terms of the sum of its parts. However, it is not the case that any other bar partition $\nu$ with $\overline{\nu}_p=\mu$ and $\overline{\nu}_q=\sigma$ `contains' $\lambda$, i.e. that $\lambda:=\{\lambda_1,\lambda_2,\dots\}$ and $\nu=\{\lambda_1+a_1,\lambda_2+a_2,\dots\}$ for some $a_1,a_2,\dots\in\mathbb{Z}_{\geq0}$; If we take $\mu=(4,1)$ and $\sigma=(3)$, then $\overline{\mu}_5=\overline{\sigma}_3$, $\lambda=(4,3,1)\in\overline{C}_{3,5}$, $\overline{\lambda}_3=\mu$, and $\overline{\lambda}_5=\sigma$, but $\nu=(13,10)$ also has 3-bar-core $\mu$ and 5-bar-core $\sigma$, and $\nu$ does not contain $\lambda$.
\end{rema}

\begin{cor}\label{cor4}
Suppose $\lambda\in\mathcal{P}_2$ is such that $|\lambda|=N$. Then $\lambda\in\overline{C}_{p,q}$ if and only if there is no $\nu\in\mathcal{P}_2\backslash\{\lambda\}$ with $|\nu|=N$, $\overline{\nu}_p=\overline{\lambda}_p$ and $\overline{\nu}_q=\overline{\lambda}_q$. 
\end{cor}
\begin{proof}
Suppose $\lambda\in\overline{C}_{p,q}$ and let $\mu=\overline{\lambda}_p$, $\sigma=\overline{\lambda}_q$. Then by Proposition \ref{prop5}, $\lambda$ is the unique smallest bar partition in $\overline{C}_{p,q}$ with $p$-bar-core $\mu$ and $q$-bar-core $\sigma$, and therefore the only one whose parts sum to $N$. 

Conversely, suppose $\lambda\not\in\overline{C}_{p,q}$. Then we can find integers $a,b,c$ such that $a\equiv b$ (mod $p$), $a\equiv c$ (mod $q$), $a,b+c-a\in\mathcal{A}(\lambda)$ and $b,c\not\in\mathcal{A}(\lambda)$ (by Proposition \ref{prop2}). 

We first assume that $a,b,c,b+c-a,-a,-b,-c$ and $a-b-c$ are distinct integers. Define a new bar partition $\nu$ by its bead configuration on the $p$-runner abacus, 
$$\mathcal{A}(\nu)=\{-a,b,c,a-b-c\}\cup\mathcal{A}(\lambda)\backslash\{a,-b,-c,b+c-a\}.$$ 
Then $\nu$ can be obtained from $\lambda$ by removing a $|b-a|$-bar and adding a $|b-a|$-bar, so $|\nu|=|\lambda|$ and since $p|(a-b)$, we have $\overline{\nu}_p=\overline{\lambda}_p$. Alternatively, we can obtain $\nu$ from $\lambda$ by removing a $|c-a|$-bar and adding a $|c-a|$-bar, so it follows from the divisibility of $a-c$ by $q$ that $\overline{\nu}_q=\overline{\lambda}_q$. Hence $\lambda$ is not the only partition with $p$-bar-core $\mu$ and $q$-bar-core $\sigma$ whose parts sum to $N$. 

If instead the integers $a,b,c$ and $b+c-a$ are not distinct, i.e., if $a=b+c-a$ or $b=c$, then they must all be congruent modulo $pq$, and $\lambda$ therefore cannot be a $pq$-bar-core. 
By adding and removing the same number of $pq$-bars, we can obtain a new partition $\nu$ from $\lambda$ with $\overline{\nu}_{pq}=\overline{\lambda}_{pq}$, $|\nu|=|\lambda|$, and $\nu\neq\lambda$. Then it is easy to see that $\overline{\nu}_p=\mu$ and $\overline{\nu}_q=\sigma$, as removing a $pq$-bar is the same as removing $p$ $q$-bars or $q$ $p$-bars. 

Now we may assume that $a,b,c,b+c-a$ are distinct but $$\{a,b,c,b+c-a\}\cap\{-a,-b,-c,a-b-c\}\neq\emptyset.$$ 
However, 
$-a=a\Rightarrow a=0\not\in\mathcal{A}(\lambda)$, a contradiction; \newline
$a-b-c=a\Rightarrow b=-c\in\mathcal{A}(\lambda)$, or $b=c=0$ and $b+c-a=-a\in\mathcal{A}(\lambda)$, both contradictions; \newline
and $a-b-c=b+c-a\Rightarrow  b+c-a=0\not\in\mathcal{A}(\lambda)$, a contradiction; \newline
so we need to consider six separate cases (or three, up to symmetry): 

(i) $a=-b$; \newline
then $-b=a\equiv b$ (mod $p)\Rightarrow p|2b\Rightarrow p|b$ and $p|a$, and $b+c-a=c-2a$. If $c-a\in\mathcal{A}(\lambda)$, then we may define $\mathcal{A}(\nu)=\{-a,c,a-c\}\cup\mathcal{A}(\lambda)\backslash\{a,-c,c-a\}$ so that we can obtain $\nu$ from $\lambda$ by removing and adding $|a|$-bars, or by removing and adding $|c-a|$-bars, and thus $\nu$ has the same size, $p$-bar-core, and $q$-bar-core as $\lambda$, since $p|a$ and $q|(a-c)$, but is distinct from $\lambda$. If instead $a-c\in\mathcal{A}(\lambda)$, defining $\mathcal{A}(\nu)=\{-a,c-a,2a-c\}\cup\mathcal{A}(\lambda)\backslash\{a,a-c,c-2a\}$, we also have $|\nu|=|\lambda|$, and we can obtain $\nu$ from $\lambda$ by adding and removing $|a|$-bars, or by adding and removing $|c-a|$-bars.

(ii) $a=-c\equiv0$ mod $q$; \newline
let $\mathcal{A}(\nu)=\begin{cases}
\{-a,b,a-b\}\cup\mathcal{A}(\lambda)\backslash\{a,-b,b-a\},&\text{if }b-a\in\mathcal{A}(\lambda);\\ 
\{-a,b-a,2a-b\}\cup\mathcal{A}(\lambda)\backslash\{a,a-b,b-2a\},&\text{if }a-b\in\mathcal{A}(\lambda).\end{cases}$ 

(iii) $b=-b$; \newline
if we let $\mathcal{A}(\nu)=\{-a,c,a-c\}\cup\mathcal{A}(\lambda)\backslash\{a,-c,c-a\}$, then we can obtain $\nu$ from $\lambda$ by removing and adding $|a|$-bars, or by removing and adding $|c-a|$-bars, and thus $\nu$ meets our criteria since $0=b\equiv a$ mod $p$.

(iv) $c=-c$; \newline
let $\mathcal{A}(\nu)=\{-a,b,a-b\}\cup\mathcal{A}(\lambda)\backslash\{a,-b,b-a\}$.

(v) $b=a-b-c$; \newline
we have $q|(a-c)=2b$, so $q|b=\tfrac{a-c}{2}$, and $p|(a-b)=\tfrac{a+c}{2}$. If $a-\tfrac{a-c}{2}=\tfrac{a+c}{2}\in\mathcal{A}(\lambda)$, then letting $\mathcal{A}(\nu)=\{\tfrac{a-c}{2},c,-\tfrac{a+c}{2}\}\cup\mathcal{A}(\lambda)\backslash\{\tfrac{c-a}{2},-c,\tfrac{a+c}{2}\}$, we have a bar partition $\nu$ that can be obtained from $\lambda$ by adding and removing $\tfrac{a-c}{2}$-bars, or by adding and removing $\tfrac{a+c}{2}$-bars. 
If $-\tfrac{a+c}{2}\in\mathcal{A}(\lambda)$, and we let $\mathcal{A}(\nu)=\{-a,\tfrac{a-c}{2},\tfrac{a+c}{2}\}\cup\mathcal{A}(\lambda)\backslash\{a,\tfrac{c-a}{2},-\tfrac{a+c}{2}\}$, then $\nu$ can be obtained from $\lambda$ by adding and removing $\tfrac{a-c}{2}$-bars, or by adding and removing $\tfrac{a+c}{2}$-bars. 

(vi) $c=\tfrac{a-b}{2}$; \newline
let $\mathcal{A}(\nu)=\begin{cases}
\{\tfrac{a-b}{2},b,-\tfrac{a+b}{2}\}\cup\mathcal{A}(\lambda)\backslash\{\tfrac{b-a}{2},-b,\tfrac{a+b}{2}\},&\text{if }\tfrac{a+b}{2}\in\mathcal{A}(\lambda);\\
\{-a,\tfrac{a-b}{2},\tfrac{a+b}{2}\}\cup\mathcal{A}(\lambda)\backslash\{a,\tfrac{b-a}{2},-\tfrac{a+b}{2}\},&\text{if }-\tfrac{a+b}{2}\in\mathcal{A}(\lambda).\end{cases}$ 

Hence, we have proved that a bar partition $\lambda$ of $N$ has $p$-bar-weight equal to the $p$-bar-weight of its $q$-bar-core precisely when there is no other bar partition of $N$ with $p$-bar-core $\overline{\lambda}_p$ and $q$-bar-core $\overline{\lambda}_q$.
\end{proof}

We will now give a method for constructing the unique bar partition $\lambda\in\overline{C}_{p,q}$ with $\overline{\lambda}_p=\mu$ and $\overline{\lambda}_q=\sigma$, for a given $p$-bar-core $\mu$ and $q$-bar-core $\sigma$ with $\overline{\mu}_q=\overline{\sigma}_p$, which we will henceforth denote by $\mu\boxplus\sigma$. In theory one can do this as in the proof of Proposition \ref{prop5}: find $a\in\mathfrak{W}_p$ such that $a\overline{\mu}_q=\mu$, and then compute $a\sigma$. But in practice, the method we give here will prove much more efficient. We will need the following lemma. 

\begin{lem}\label{lem2} 
Suppose that $\lambda\in\overline{C}_q$ and $j,k\not\equiv0\:(\text{mod }p)$. Then 
$$\Delta_{j\text{ mod }p}\lambda\equiv\Delta_{k\text{ mod }p}\lambda\:(\text{mod }q)\Rightarrow\lambda^{(j\text{ mod }p)}=\lambda^{(k\text{ mod }p)}.$$
\end{lem}
\begin{proof}
Without loss of generality, suppose $\Delta_k\lambda>\Delta_j\lambda$. We have that $$\{x\in\mathcal{A}(\lambda)|x\equiv j\text{ (mod }p)\}=\{p(\lambda^{(j)}_i-i)+\Delta_j\lambda|i\in\mathbb{N}\},$$ where $\lambda^{(j)}_i=0$ for $i>l$, if $l$ is the number of (positive) parts of $\lambda^{(j)}$, since $$r:=|\{x\in\lambda|x\equiv j\text{ (mod }p)\}|-|\{x\in\lambda|x\equiv-j\text{ (mod }p)\}|=\tfrac{\Delta_j\lambda-j}{p}.$$ If $\lambda^{(j)}\neq\lambda^{(k)}$, then $\mathcal{B}^{\lambda^{(k)}}_0\not\subseteq\mathcal{B}^{\lambda^{(j)}}_0$, by the following result \cite[Lemma 2.1]{F1}: 
\begin{center}Suppose $\tau$ and $\rho$ are partitions and $r\in\mathbb{Z}$. If $\mathcal{B}^\tau_r\subseteq\mathcal{B}^\rho_r$, then $\tau=\rho$. 
\end{center}
The proof of this is simple: Choose $N$ sufficiently large that $\tau_N=\rho_N=0$. Then $\mathcal{B}^\tau_r$ and $\mathcal{B}^\rho_r$ both contain all integers less than or equal to $r-N$. Moreover, $\mathcal{B}^\tau_r$ contains exactly $N-1$ elements greater than $r-N$, namely $\tau_1-1+r,\dots,\tau_{N-1}-(N-1)+r$, and similarly $\mathcal{B}^\rho_r$ contains exactly $N-1$ elements greater than $r-N$. Since $\mathcal{B}^\tau_r\subseteq\mathcal{B}^\rho_r$, we get $\mathcal{B}^\tau_r=\mathcal{B}^\rho_r$, and hence $\tau=\rho$. 

Thus there is an $i$ with $\lambda^{(k)}_i-i\not\in\mathcal{B}^{\lambda^{(j)}}_0$, or equivalently, $\lambda^{(k)}_i-i+r\not\in\mathcal{B}^{\lambda^{(j)}}_r$, and it follows that $m(\lambda^{(k)}_i-i)+\Delta_j\lambda\not\in\mathcal{A}(\lambda)$. Since $m(\lambda^{(k)}_i-i)+\Delta_k\lambda\in\mathcal{A}(\lambda)$, $\lambda$ cannot be a $(\Delta_k\lambda-\Delta_j\lambda)$-bar-core and is therefore not a $q$-bar-core; a contradiction.
\end{proof} 

Note that since $\Delta_0\lambda=0$ for all bar partitions $\lambda$, it is not necessary that $\lambda^{(0)}$ and $\lambda^{(j)}$ are equal when $\Delta_j\lambda\equiv0$ (mod $q$). 

\begin{prop}\label{prop6}
Suppose $\mu\in\overline{C}_p$ and $\sigma\in\overline{C}_q$ are such that $\overline{\mu}_q=\overline{\sigma}_p$. Then there is a unique $\lambda\in\mathcal{P}_2$ with $p$-bar-core $\mu$ and with the same $q$-weighted $p$-quotient as $\sigma$. 
\end{prop}
\begin{proof}
Let $\nu=\overline{\mu}_q$. Then by Proposition \ref{prop4} $\mu$ and $\nu$ lie in the same level $q$ orbit of $\mathfrak{W}_p$, so there is a permutation $\phi$ on $\{0,\dots,p-1\}$ such that $\Delta_j\mu\equiv\Delta_{\phi(j)}\nu$ (mod $n$) for each $j\in\{0,\dots,p-1\}$. Thus by Lemma \ref{lem2} $\nu$ must have the same $q$-weighted $p$-quotient as $\mu$. Since $\nu=\overline{\sigma}_p$, we may construct $\lambda$ by taking the bar partition with $p$-bar-core $\mu$ and $\lambda^{(j)}=\sigma^{(\phi(j))}$ for each $j\in\{0,\dots,p-1\}$. 

By Lemma \ref{lem2} we have $\sigma^{(j)}=\sigma^{(k)}$ whenever $\Delta_j\sigma\equiv\Delta_k\sigma$ (mod $q$), i.e. whenever $\Delta_{\phi^{-1}(j)}\mu\equiv\Delta_{\phi^{-1}(k)}\mu$ (mod $q$), so the $\lambda$ we construct is unique (as it is uniquely determined by its $p$-bar-core and $p$-quotient). 
\end{proof}

\begin{prop}\label{prop10}
Suppose $\mu\in\overline{C}_p$ and $\sigma\in\overline{C}_q$ are such that $\overline{\mu}_q=\overline{\sigma}_p$, and let $\lambda$ be the bar partition with $p$-bar-core $\mu$ and with the same $q$-weighted $p$-quotient as $\sigma$. Then $\lambda=\mu\boxplus\sigma$. 
\end{prop}
\begin{proof}
By Proposition \ref{prop8} $\lambda$ and $\sigma$ lie in the same level $q$ orbit of $\mathfrak{W}_p$. Hence it follows from Proposition \ref{prop4} that $\lambda\in\overline{C}_{p,q}$ and $\overline{\lambda}_q=\sigma$. 
\end{proof}

\begin{exam}
\begin{align*}\mu=(&27,22,17,12,7,4,2)&&&&\sigma=(11,8,5,2)\\
&\abacus(vvvvv,bbbbb,nbbbb,nbbbb,nbbbb,nbbbb,nbbnb,nbonb,nbnnb,nnnnb,nnnnb,nnnnb,nnnnb,nnnnn,vvvvv)
&&\overline{\mu}_3=(1)=\overline{\sigma}_5&&
\abacus(vvvvv,bbbbb,bbbbb,bbbbb,bbbbb,bnbbn,bbnbb,nbonb,nnbnn,bnnbn,nnnnn,nnnnn,nnnnn,nnnnn,vvvvv)
\end{align*}
\begin{align*}
\mathcal{Q}^3_5(\mu)=[(0,\varnothing),(2,\varnothing),(2,\varnothing),(0,\varnothing),(0,\varnothing)]\\
{Q}^3_5(\sigma)=&[(0,(1)),(0,(2)),(2,(1^2)),(0,(2)),(2,(1^2))]
\end{align*}
Using our algorithm, we will obtain the bead configuration for $\mu\boxplus\sigma$: \newline
We first replace the 0-runner with the 0-runner from the configuration for $\sigma$. Next, we have $\Delta_{1\text{ mod }5}\sigma\equiv0$ (mod 3) so for each runner $j\in\{1,\dots,4\}$ such that $\Delta_{j\text{ mod }5}\mu\equiv0$ (mod 3), in this case runners 3 and 4, since $\sigma^{(1\text{ mod }5)}=(2)$, we move the lowest bead on the runner upwards 2 spaces. Finally, $\Delta_{2\text{ mod }5}\sigma\equiv2\equiv\Delta_{1\text{ mod }5}\mu\equiv\Delta_{2\text{ mod }5}\mu$ (mod 3) and $\sigma^{(2\text{ mod }5)}=(1^2)$, so we move the second lowest beads up by two spaces on runners 1 and 2.
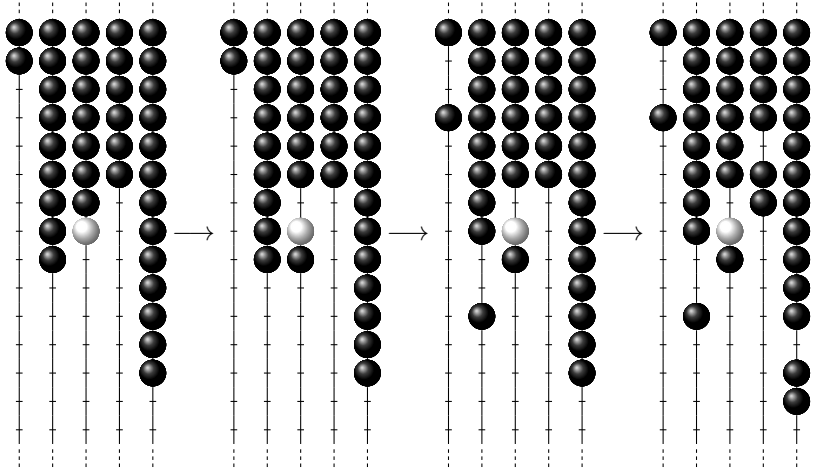
\begin{center}\abacus(vvvvv,bbbbb,bbbbb,nbbbb,nbbbb,nbbbb,nbbbb,nbbnb,nbonb,nbnnb,nnnnb,nnnnb,nnnnb,nnnnb,nnnnn,nnnnn,vvvvv)\:\raisebox{3.1cm}{$\longrightarrow$}\:\abacus(vvvvv,bbbbb,bbbbb,nbbbb,nbbbb,nbbbb,nbbbb,nbnnb,nbonb,nbbnb,nnnnb,nnnnb,nnnnb,nnnnb,nnnnn,nnnnn,vvvvv)\:\raisebox{3.1cm}{$\longrightarrow$}\:\abacus(vvvvv,bbbbb,nbbbb,nbbbb,bbbbb,nbbbb,nbbbb,nbnnb,nbonb,nnbnb,nnnnb,nbnnb,nnnnb,nnnnb,nnnnn,nnnnn,vvvvv)\:\raisebox{3.1cm}{$\longrightarrow$}\:\abacus(vvvvv,bbbbb,nbbbb,nbbbb,bbbbb,nbbnb,nbbbb,nbnbb,nbonb,nnbnb,nnnnb,nbnnb,nnnnn,nnnnb,nnnnb,nnnnn,vvvvv)
\end{center}
Thus we obtain the bar partition $\mu\boxplus\sigma=(32,27,17,14,12,7,5,2)$.
\end{exam}

\section{The $\Upsilon$-orbit}
\noindent
In \cite{BO}, Bessenrodt \& Olsson showed there is a maximal bar partition $\Upsilon_{\text{min}\{p,q\},\text{max}\{p,q\}}$ which is both a $p$-bar-core and a $q$-bar-core, where $\Upsilon_{p,q}$ is the \textbf{Yin/Yang partition}, with parts 
$$(\tfrac{p-1}{2}-k)q-(l+1)p\text{, for }k,l\in\mathbb{Z}_{\geq0}.$$ 
It is maximal in the sense that whenever $\lambda:=\{\lambda_1,\lambda_2,\dots\}$ is both a $p$-bar-core and a $q$-bar-core, we may write $\Upsilon_{\text{min}\{p,q\},\text{max}\{p,q\}}=\{\lambda_1+a_1,\lambda_2+a_2,\dots\}$ with $a_1,a_2,\dots\in\mathbb{Z}_{\geq0}$. 
When $p<q$, $\Upsilon_{p,q}$ is called the Yin partition, and $\Upsilon_{q,p}$ the Yang partition. 

Since $\Upsilon_{p,q}\in\overline{C}_p\cap\overline{C}_q$, its $p$-set is $\{0,q,\dots,\nicefrac{q(p-1)}{2}\}\cup\{p-q,p-2q,\dots,p-\nicefrac{q(p-1)}{2}\}$, and its $q$-set is $\{0\}\cup\{\nicefrac{q(p+1)}{2}-kp|k=1,\dots,q-1\}$. Thus, the $p$-set of $\Upsilon_{p,q}$ consists of $\nicefrac{(p+1)}{2}$ elements divisible by $q$ and $\nicefrac{(p-1)}{2}$ elements congruent to $p$ modulo $q$, while the $q$-set of $\Upsilon_{p,q}$ consists of $0$ and $q-1$ integers congruent to $\nicefrac{q(p+1)}{2}$ modulo $p$. 

\begin{exam}
The bar partition 
$$\Upsilon_{5,11}=(17,12,7,6,2,1)$$ 
has $5$-set $\{0,11,22,-17,-6\}$, and $11$-set $\{0,23,13,3,-7,-17,28,18,8,-2,-12\}$, while 
$$\Upsilon_{11,5}=(14,9,4,3)$$ 
has $5$-set $\{0,-14,-3,8,19\}$, and $11$-set $\{0,1,-9,25,15,5,6,-4,-14,20,10\}$.
\end{exam}

To conclude this paper, we will consider $\overline{C}^\Upsilon_{p,q}$, the $\Upsilon$\textbf{-orbit} of the set $\overline{C}_{p,q}$, which we define to be the orbit of $\Upsilon_{p,q}$ under the action of $\mathfrak{W}_p\times\mathfrak{W}_q$. Our final result will establish a bijection 
$$\overline{C}^\Upsilon_{p,q}\to2^{\{1,\dots,\nicefrac{(p-1)}{2}\}}\times\overline{C}_p\times\overline{C}_q.$$ 

\begin{lem}\label{sig2}
Suppose that $\sigma\in\overline{C}_q$. If $\overline{\sigma}_p=\Upsilon_{p,q}$, then $\sigma^{(0\text{ mod }p)}$ is a $q$-bar-core and there is a $q$-core $\delta$ such that 
$$\sigma^{(j\text{ mod }p)}=\begin{cases}\delta,&\text{ for }j\equiv kq,\:k\in\{1,\dots,\tfrac{p-1}{2}\},\\
\delta',&\text{ for }j\equiv kq,\:k\in\{\tfrac{p+1}{2},\dots,p-1\}.\end{cases}$$ 
If instead $\overline{\sigma}_p=\Upsilon_{q,p}$, then $\sigma^{(0\text{ mod }p)}$ is a $q$-bar-core and there is a self-conjugate $q$-core $\gamma$ such that $\sigma^{(j\text{ mod }p)}=\gamma=\gamma'$, for all $j\not\equiv0\:(\text{mod }p)$. 
\end{lem}
\begin{proof}
When $\overline{\sigma}_p=\Upsilon_{p,q}$, the non-zero elements of the shared $p$-set of $\Upsilon_{p,q}$ and $\sigma$ belong to two congruence classes modulo $p$, so by Lemmas \ref{lem3}(1) and \ref{lem2}, $\mathcal{Q}_p(\sigma)$ consists of a $q$-bar-core $\sigma^{(0\text{ mod }p)}$ and at most two distinct $q$-cores. Moreover, since $(\sigma^{(j\text{ mod }p)})'=\sigma^{(-j\text{ mod }p)}$ for each $j\not\equiv0\:($mod $p)$, the $p$-quotient of $\sigma$ consists of $\sigma^{(0\text{ mod }p)}$ (the parts of $\sigma$ which are multiples of $p$, divided by $p$) and either $p-1$ other empty bar partitions (when $\sigma=\Upsilon_{p,q}$), or $\nicefrac{(p-1)}{2}$ copies each of two conjugate partitions. 

When $\overline{\sigma}_p=\Upsilon_{q,p}$, all of the non-zero elements in the $p$-set of $\sigma$ are congruent modulo $p$, so again by Lemmas \ref{lem3}(1) and \ref{lem2}, the $q$-quotient $\mathcal{Q}_q(\sigma)$ simply consists of a $q$-bar-core $\sigma^{(0\text{ mod }p)}$ and $p-1$ copies of a self-conjugate $q$-core. 
\end{proof}

The above lemma means that the construction of $\mu\boxplus\sigma$ becomes even more straightforward when $\mu$ and $\sigma$ are contained in the $\Upsilon$-orbit. 

\begin{prop}\label{sig1}
Suppose $\mu\in\overline{C}_p$ and $\sigma\in\overline{C}_q$ are such that $\overline{\mu}_q=\overline{\sigma}_p=\Upsilon_{p,q}$. Then $\mu\boxplus\sigma$ is the bar partition $\lambda$ with $\overline{\lambda}_p=\mu$, $\lambda^{(0\text{ mod }p)}=\sigma^{(0\text{ mod }p)}$, and $$\lambda^{(j\text{ mod }p)}=\begin{cases}\sigma^{(j\text{ mod }p)}&\text{ if }\Delta_{j\text{ mod }p}\mu\equiv\Delta_{j\text{ mod }p}\Upsilon_{p,q}\:(\text{mod }q),\\(\sigma^{(j\text{ mod }p)})'&\text{ otherwise.}\end{cases}$$ 
Moreover, $\mu\boxplus\sigma$ is also the bar partition with $q$-bar-core $\sigma$ and the same $q$-quotient as $\mu$.  
\end{prop}
\begin{proof}
There are $\nicefrac{(p+1)}{2}$ elements in the $p$-set of both $\Upsilon_{p,q}=\overline{\sigma}_p$ and $\sigma$ that are divisible by $q$, and the other $\nicefrac{(p-1)}{2}$ elements are congruent to $p$ modulo $q$. Hence, it follows from Proposition \ref{prop6} that $\mu\boxplus\sigma=\lambda$. 
The elements $\Delta_{1\text{ mod }q}\sigma,\dots,\Delta_{p-1\text{ mod }q}\sigma$ of the $q$-set of $\sigma$ are all congruent modulo $p$, so the $p$-quotients of the two bar partitions $\sigma$ and $\mu\boxplus\sigma$ are exactly the same. 
\end{proof}

\begin{exam}
When $\mu=(21,16,11,7,6,2,1)\in\overline{C}_5$ and $\sigma=(19,12,5,4)\in\overline{C}_7$, so that $\overline{\mu}_7=\overline{\sigma}_5=(9,4,2)=\Upsilon_{5,7}$, we find that the $5$-set of $\mu$ and $\mu\boxplus\sigma$ is $\{0,26,12,-7,-21\}$, the $5$-set of $\Upsilon_{5,7}$ and $\sigma$ is $\{0,-9,7,-2,14\}$, \newline
$\mathcal{Q}^7_5(\Upsilon_{5,7})=[(0,\varnothing),(5,\varnothing),(0,\varnothing),(5,\varnothing),(0,\varnothing)]$, \newline
$\mathcal{Q}^7_5(\mu)=[(0,\varnothing),(5,\varnothing),(5,\varnothing),(0,\varnothing),(0,\varnothing)]$, \newline
$\mathcal{Q}^7_5(\sigma)=[(0,(1)),(5,(1^2)),(0,(2)),(5,(1^2)),(0,(2))]$, and \newline
$\mathcal{Q}^7_5(\mu\boxplus\sigma)=[(0,(1)),(5,(1^2)),(5,(1^2)),(0,(2)),(0,(2))]$; \newline
using our algorithm we obtain 
$$\mu\boxplus\sigma=(26,21,12,11,7,6,5,1).$$
But with $\nu=(21,16,11,6,2,1)=\overline{\nu}_5$, so that $\overline{\nu}_7=\Upsilon_{5,7}$, and $\sigma=(19,12,5,4)$ again, we find that the $5$-set of $\nu$ and $\nu\boxplus\sigma$ is $\{0,26,7,-2,-21\}$, \newline
$\mathcal{Q}^7_5(\nu)=[(0,\varnothing),(5,\varnothing),(0,\varnothing),(5,\varnothing),(0,\varnothing)]=\mathcal{Q}^7_5(\Upsilon_{5,7})$, and \newline
$\mathcal{Q}^7_5(\sigma)=[(0,(1)),(5,(1^2)),(0,(2)),(5,(1^2)),(0,(2))]=\mathcal{Q}^7_5(\nu\boxplus\sigma)$, so we instead get 
$$\nu\boxplus\sigma=(26,21,12,11,6,5,1).$$
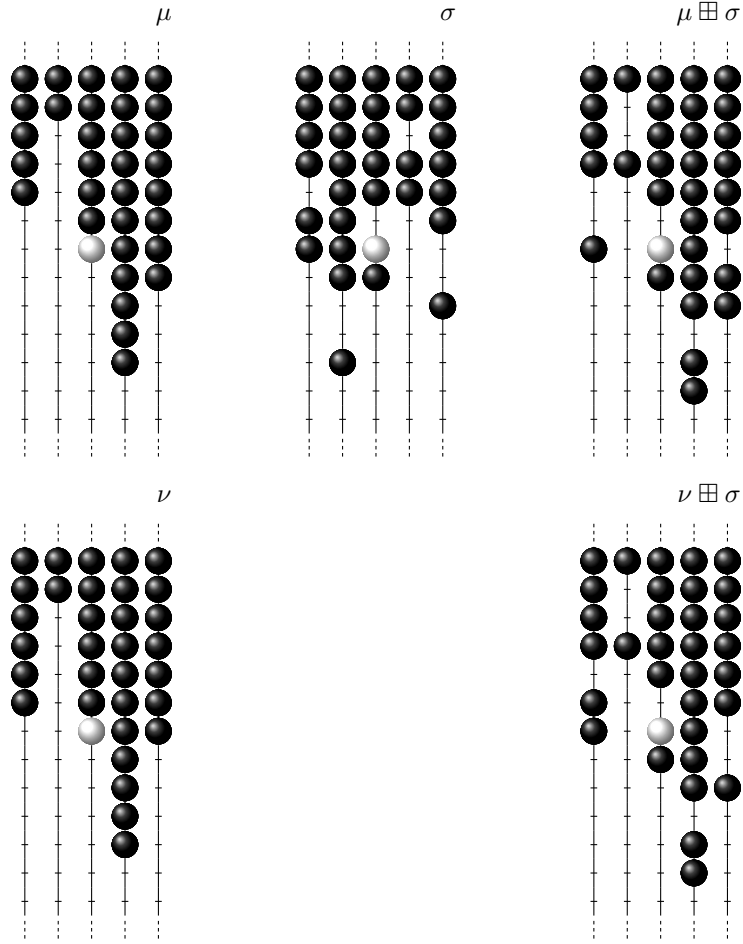
\begin{align*}\mu&&\sigma&&\mu\boxplus\sigma\\
\abacus(vvvvv,bbbbb,bbbbb,bnbbb,bnbbb,bnbbb,nnbbb,nnobb,nnnbb,nnnbn,nnnbn,nnnbn,nnnnn,nnnnn,vvvvv)&&
\abacus(vvvvv,bbbbb,bbbbb,bbbnb,bbbbb,nbbbb,bbnnb,bbonn,nbbnn,nnnnb,nnnnn,nbnnn,nnnnn,nnnnn,vvvvv)&&
\abacus(vvvvv,bbbbb,bnbbb,bnbbb,bbbbb,nnbbb,nnnbb,bnobn,nnbbb,nnnbb,nnnnn,nnnbn,nnnbn,nnnnn,vvvvv)\\
\nu&&&&\nu\boxplus\sigma\\
\abacus(vvvvv,bbbbb,bbbbb,bnbbb,bnbbb,bnbbb,bnbbb,nnobb,nnnbn,nnnbn,nnnbn,nnnbn,nnnnn,nnnnn,vvvvv)&&&&
\abacus(vvvvv,bbbbb,bnbbb,bnbbb,bbbbb,nnbbb,bnnbb,bnobn,nnbbn,nnnbb,nnnnn,nnnbn,nnnbn,nnnnn,vvvvv)
\end{align*}

Next we will consider the $5$-weighted $7$-quotient, so whichever $5$-bar-core $\mu$ and $7$-bar-core $\sigma$ with $\overline{\mu}_7=\overline{\sigma}_5=\Upsilon_{5,7}$ we choose, we will find that the $7$-set of $\sigma$ contains $\Delta_0\sigma=0$ and $6$ elements in the same conjugacy class modulo $5$. This is because $\overline{\mu}_7=\Upsilon_{5,7}$ so that the corresponding elements in the $7$-sets of the two bar partitions can only differ by multiples of $5$, and as discussed above, $\Upsilon_{5,7}$ has a $7$-set of this form. \newline
Let $\mu=(16,11,7,6,2,1)=\overline{\mu}_5$ and $\sigma=(19,12,5,4)=\overline{\sigma}_7$, so that $\overline{\mu}_7=\overline{\sigma}_5=(9,4,2)=\Upsilon_{5,7}$. The $7$-set of $\Upsilon_{m,n}$ and $\mu$ is $\{0,1,16,-4,11,-9,6\}$, the $7$-set of $\sigma$ and $\mu\boxplus\sigma$ is $\{0,1,-19,-4,11,26,6\}$, \newline
$\mathcal{Q}^5_7(\Upsilon_{5,7})=[(0,\varnothing),(1,\varnothing),(1,\varnothing),(1,\varnothing),(1,\varnothing),(1,\varnothing),(1,\varnothing)]=\mathcal{Q}^5_7(\sigma)$, and \newline
$\mathcal{Q}^5_7(\mu)=[(0,(1)),(1,(1)),(1,(1)),(1,(1)),(1,(1)),(1,(1)),(1,(1))]=\mathcal{Q}^5_7(\mu\boxplus\sigma)$, \newline
so we obtain 
$$\mu\boxplus\sigma=(26,12,11,7,6,5,1).$$ 
With $\sigma=(19,12,5,4)$ again and $\nu=(23,18,13,9,8,4,3)=\overline{\nu}_5$, so that $\overline{\nu}_7=\Upsilon_{5,7}$, we find the $7$-set of $\Upsilon_{m,n}$ and $\nu$ is $\{0,1,16,-4,11,-9,6\}$, \newline
the $7$-set of $\nu\boxplus\sigma$ is again $\{0,1,-19,-4,11,26,6\}$, \newline
$\mathcal{Q}^5_7(\sigma)=[(0,\varnothing),(1,\varnothing),(1,\varnothing),(1,\varnothing),(1,\varnothing),(1,\varnothing),(1,\varnothing)]$, and $\mathcal{Q}^5_7(\nu\boxplus\sigma)=$ \newline
$\mathcal{Q}^5_7(\nu)=[(0,(2,1)),(1,(1,(2,1)),(1,(2,1)),(1,(2,1)),(1,(2,1)),(1,(2,1)),(1,(2,1))]$, \newline
so we obtain 
$$\nu\boxplus\sigma=(33,19,18,13,8,5,4,3).$$
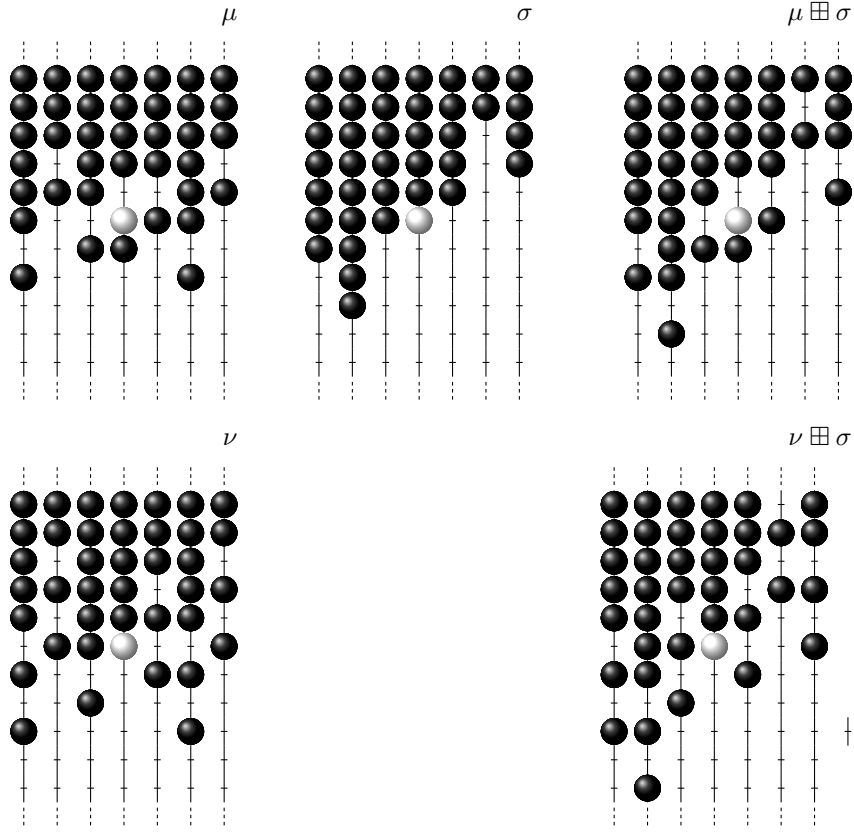
\begin{align*}\mu&&\sigma&&\mu\boxplus\sigma\\
\abacus(vvvvvvv,bbbbbbb,bbbbbbb,bbbbbbb,bnbbbbn,bbbnnbb,bnnobbn,nnbbnnn,bnnnnbn,nnnnnnn,nnnnnnn,nnnnnnn,vvvvvvv)&&
\abacus(vvvvvvv,bbbbbbb,bbbbbbb,bbbbbnb,bbbbbnb,bbbbbnn,bbbonnn,bbnnnnn,nbnnnnn,nbnnnnn,nnnnnnn,nnnnnnn,vvvvvvv)&&
\abacus(vvvvvvv,bbbbbbb,bbbbbnb,bbbbbbb,bbbbbnn,bbbnnnb,bbnobnn,nbbbnnn,bbnnnnn,nnnnnnn,nbnnnnn,nnnnnnn,vvvvvvv)\\
\nu&&&&\nu\boxplus\sigma\\
\abacus(vvvvvvv,bbbbbbb,bbbbbbb,bnbbbbn,bbbbnbb,bnbbbbn,nbbonnb,bnnnbbn,nnbnnnn,bnnnnbn,nnnnnnn,nnnnnnn,vvvvvvv)&&&&
\abacus(vvvvvvv,bbbbbnb,bbbbbbb,bbbbbnn,bbbbnbb,bbnbbnn,nbbonnb,bbnnbnn,nnbnnnn,bbnnnnnn,nnnnnnn,nbnnnnn,vvvvvvv)
\end{align*}
\end{exam}

The next result gives the converse to Lemma \ref{sig2}, establishes that $q$-bar-cores $\sigma$ are uniquely determined by $\sigma^{(0\text{ mod }p)}$ when $\overline{\sigma}_p=\Upsilon_{p,q}$, and gives the number of $q$-bar-cores $\mu$ with $p$-bar-core $\Upsilon_{q,p}$ when $\mu^{(0\text{ mod }p)}$ is fixed. First it is necessary to recall the definition of the beta-set $\mathcal{B}^\alpha_r:=\{\alpha_i-i+r|i\in\mathbb{Z}_{>0}\}$ of a partition $\alpha$, and define the \textit{double} of a bar partition. 

\begin{defn}
The \textbf{double} of a bar partition $\lambda=(\lambda_1,\dots,\lambda_r)$ is the partition $D(\lambda)$ whose Young diagram is obtained by amalgamating the \textit{shifted Young diagram} of $\lambda$, which has $\lambda_i$ nodes in the $i^\text{th}$ row, with the left-most in the $i^\text{th}$ column, and its reflection along the top left to bottom right diagonal. Equivalently, $[D(\lambda)]$ is the union of $[(\lambda_1+1,\lambda_2+2,\dots,\lambda_r+r)]$ and $[(\lambda_1,\lambda_2+1,\dots,\lambda_r+r-1)']$. \cite[p. 14]{MD}
\end{defn}

There is a correspondence between $p$-bars of $\lambda\in\mathcal{P}_2$ and $p$-hooks of the partition $D(\lambda)$, which means that when $\lambda$ is a $p$-bar-core, $D(\lambda)$ must be a $p$-core. 

\begin{exam}
$D((7,4,3,2))=(8,6,6,6,4,1,1)$:
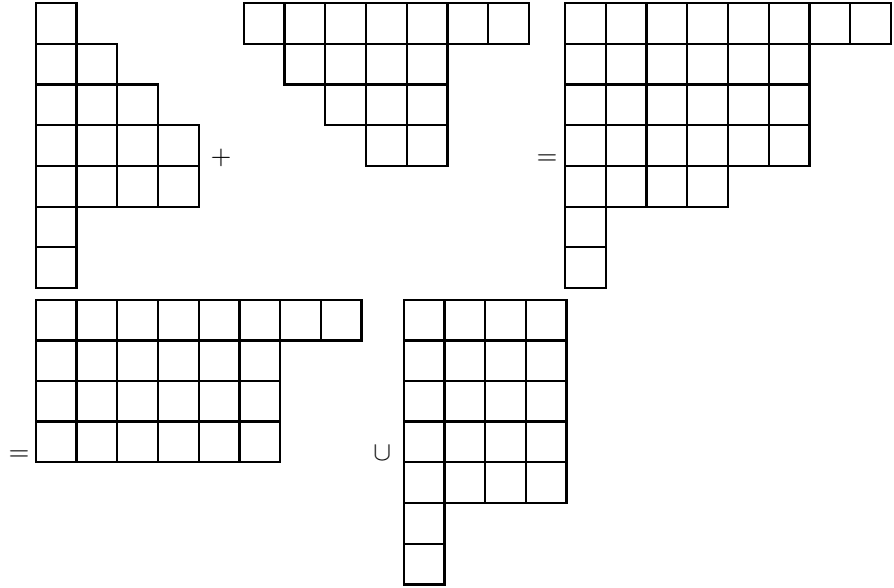
\begin{align*}
&\raisebox{1.6cm}{\ydiagram{1,2,3,4,4,1,1}}\:+\:\raisebox{1.6cm}{\ydiagram{7,1+4,2+3,3+2}}=\raisebox{1.6cm}{\ydiagram{8,6,6,6,4,1,1}}\\
=\:&\raisebox{1.6cm}{\ydiagram{8,6,6,6}}\:\cup\:\raisebox{1.6cm}{\ydiagram{4,4,4,4,4,1,1}}
\end{align*}
\end{exam}

\begin{prop}\label{prop9}
\begin{enumerate}
\item Suppose $\mu\in\mathcal{P}_2$ and $\overline{\mu}_q=\Upsilon_{p,q}$. Then $\mu\in\overline{C}_p$ if and only if $\alpha:=\mu^{(0\text{ mod }q)}$ is a $p$-bar-core, there is a self-conjugate $p$-core $\gamma$ with $\gamma=\mu^{(j\text{ mod }q)}$ for all $j\not\equiv0\:($mod $q)$, and 
$$\mathcal{B}^\gamma_{\nicefrac{(1-p)}{2}}\subseteq\mathcal{A}(\alpha)=\{\tfrac{x}{q}|x\in\mathcal{A}(\mu)\cap q\mathbb{Z}\}\subseteq\mathcal{B}^\gamma_{\nicefrac{(p+1)}{2}}.$$ 
Moreover, for each $\alpha\in\mathcal{P}_2$, there are $2^{\nicefrac{(p-1)}{2}}$ $p$-bar-cores $\mu$ with $\overline{\mu}_q=\Upsilon_{p,q}$ and $\mu^{(0\text{ mod }q)}=\alpha$. 
\item Suppose $\sigma\in\mathcal{P}_2$ and $\overline{\sigma}_p=\Upsilon_{p,q}$. Then 
$$\sigma\in\overline{C}_q\Leftrightarrow\exists\beta\in\overline{C}_q,\:\sigma^{(j\text{ mod }p)}=\begin{cases}\beta&\text{if }j\equiv0\:(\text{mod }p),\\
D(\beta)&\text{if }j\equiv iq\:(\text{mod }p)\text{ for some }1\leq i\leq\tfrac{p-1}{2},\\
D(\beta)'&\text{otherwise.}\end{cases}$$ 
\end{enumerate}
\end{prop}
\begin{proof}
(1) First suppose that $\mathcal{Q}_q(\mu)$ consists of a $p$-bar-core $\alpha:=\mu^{(0\text{ mod }q)}$ and $q-1$ copies of a self-conjugate $p$-core $\gamma$ such that 
$$\mathcal{B}^\gamma_{\nicefrac{(1-p)}{2}}\subseteq\mathcal{A}(\alpha)\subseteq\mathcal{B}^\gamma_{\nicefrac{(p+1)}{2}}.$$ 
Since the $q$-set of $\Upsilon_{p,q}$ and $\mu$ is $\{0\}\cup\{\nicefrac{q(p+1)}{2}-np|n=1,2,\dots,q-1\}$, for each $n\in\{1,\dots,q-1\}$ we have 
$$\{x\in\mathcal{A}(\mu)|x\equiv np\:(\text{mod }q)\}=\{(\gamma_i-i-\tfrac{p-1}{2})q+np|i\in\mathbb{N}\},$$ 
so clearly $x-p\in\mathcal{A}(\mu)$ whenever $n\in\{2,\dots,q-1\}$, $x\in\mathcal{A}(\mu)$, $x\equiv np$ (mod $q$). Furthermore, from the assumed inclusions involving $\mathcal{A}(\alpha)$ and beta-sets for $\gamma$, 
\begin{align*}
x\in\mathcal{A}(\mu)\cap q\mathbb{Z}\Rightarrow\exists j,x-p&=(\gamma_j-j+\tfrac{p+1}{2})q-p\\
&=(\gamma_j-j-\tfrac{p-1}{2})q+(q-1)p\in\mathcal{A}(\mu);\\
x\in\mathcal{A}(\mu),x\equiv p\:(\text{mod }q)\Rightarrow\exists i,x-p&=(\gamma_i-i-\tfrac{p-1}{2})q+p-p\\
&=(\gamma_i-i+\tfrac{1-p}{2})q\in\mathcal{A}(\mu);
\end{align*}
hence $\overline{\mu}_p=\mu$. 

Now suppose $\mu\in\overline{C}_p$. Then, again by Lemma \ref{sig2}, $\mathcal{Q}_q(\mu)$ consists of a $p$-bar-core $\alpha=\mu^{(0\text{ mod }q)}$ and a self-conjugate $p$-core $\gamma$, so 
$$\{x\in\mathcal{A}(\mu)|x\equiv np\:(\text{mod }q)\}=\{(\gamma_i-i-\nicefrac{(p-1)}{2})q+np|i\in\mathbb{N}\}$$ for each $n\in\{1,\dots,q-1\}$. Thus 
\begin{align*}
&\forall k\in\mathbb{N},\alpha_kq-p\in\mathcal{A}(\mu)\Rightarrow\exists j,\alpha_k=\gamma_j-j+\tfrac{p+1}{2};\\
&\forall i\in\mathbb{N},(\gamma_i-i-\tfrac{p-1}{2})q\in\mathcal{A}(\mu)\Rightarrow\exists h,\gamma_i-i+\tfrac{1-p}{2}=\alpha_h; \end{align*}
hence 
$$\mathcal{B}^\gamma_{\nicefrac{(1-p)}{2}}\subseteq\mathcal{A}(\alpha)\subseteq\mathcal{B}^\gamma_{\nicefrac{(p+1)}{2}}.$$ 
Let $z$ be the largest element of $\mathcal{A}(\alpha)$ (i.e. $z:=-1$ if $\alpha=\varnothing$, or $z:=\alpha_1$ is the largest part of $\alpha$, otherwise). If $\mu\in\overline{C}_p$, then for each $m\in\mathbb{N}$ and $n\in\{0,1,\dots,q-1\}$ we have $(z+m)q+np\not\in\mathcal{A}(\mu)$ (and $-(z+m)q-np\in\mathcal{A}(\mu)$). Since $\overline{\mu}_q=\Upsilon_{p,q}$, $\Delta_{np\text{ mod }q}\mu=\nicefrac{q(p+1)}{2}-(q-n)p$ for $n\in\{1,\dots,q-1\}$, so the number of integers $x\in\mathcal{A}(\mu)$ such that $x\equiv np$ (mod $q$) and $-zq-(q-n)p<x<(z+1)q+np$ is $\tfrac{1}{q}(\nicefrac{q(p+1)}{2}-(q-n)p-(-zq-(q-n)p))=z+\nicefrac{(p+1)}{2}$. 
The partition $\gamma$ must be self-conjugate, so for each $m\in\mathbb{Z}$ and $n\in\{1,\dots,q-1\}$, we have 
$$(z+1)q+np-(z+\tfrac{p+1}{2})q+m=np-\tfrac{p-1}{2}q+m\in\mathcal{A}(\mu)\Leftrightarrow np-\tfrac{p+1}{2}q-m\not\in\mathcal{A}(\mu).$$ 
Moreover, we must have $yq-p\in\mathcal{A}(\mu)$ (and $p-yq\not\in\mathcal{A}(\mu)$) for all $y\in\mathcal{A}(\alpha)\cup\{0\}$, so for each $n\in\{1,\dots,q-1\}$, there are $z+1$ elements in $\{x\in\mathcal{A}(\mu)|x\equiv np$ (mod $q)\}$ which are fixed by $\alpha$, and $z+\nicefrac{(p+1)}{2}-(z+1)=\nicefrac{(p-1)}{2}$ integers which are free to either belong in this set or not. Hence the number of possible partitions $\gamma$, and therefore the number of $p$-bar-cores $\mu$ with $q$-bar-core $\Upsilon_{p,q}$ and $\mu^{(0\text{ mod }q)}=\alpha$, is $2^{\nicefrac{(p-1)}{2}}$. 

\vspace{3mm}

(2) Next, suppose $\beta:=\sigma^{(0\text{ mod }p)}$ is a $q$-bar-core, and for $j\not\equiv0\:($mod $p)$, 
$$\sigma^{(j\text{ mod }p)}=\begin{cases}D(\beta)&\text{ if } j\equiv iq\:(\text{mod }p)\text{ for some }1\leq i\leq\tfrac{p-1}{2},\\D(\beta)'&\text{ otherwise.}\end{cases}$$ 
Since the $p$-set of $\Upsilon_{p,q}$ and $\sigma$ is $\{nq|n=0,1,\dots,\nicefrac{(p-1)}{2}\}\cup\{p-nq|n=1,2,\dots,\nicefrac{(p-1)}{2}\}$, for each $n\in\{1,\dots,\nicefrac{(p-1)}{2}\}$, (adopting the convention that for a partition $\lambda$, we put $\lambda_i=0$ for all $i>\lambda_1'$) we have 
\begin{align*}\{x\in\mathcal{A}(\sigma)|x\equiv nq\text{ (mod }p)\}&=\{(D(\beta)_i-i)p+nq|i\in\mathbb{N}\};\\
\{x\in\mathcal{A}(\sigma)|x\equiv-nq\text{ (mod }p)\}&=\{(D(\beta)_j'-j+1)p-nq|j\in\mathbb{N}\}.\end{align*}
It follows that $x-q\in\mathcal{A}(\sigma)$ whenever $x\in\mathcal{A}(\sigma)$, $n\in\{2,\dots,p-1\}\backslash\{\nicefrac{(p+1)}{2}\}$ and $x\equiv nq$ (mod $p$). We also have $x-q\in\mathcal{A}(\sigma)$ for each $x\equiv0,q$ (mod $p$) by the definition of $D(\beta)$: for each part $\beta_k$ of $\beta=(\beta_1,\dots,\beta_r)$, $D(\beta)_k'-k+1=\beta_k=D(\beta)_k-k$. Since $\beta$ is a $q$-bar-core, $D(\beta)$ is a $q$-core, so when $1\leq k\leq r$ we have 
$$D(\beta)_k'-k+1-q=\beta_k-q=D(\beta)_k-k-q\in\mathcal{B}^{D(\beta)}.$$ 
It therefore follows from the symmetry of $\mathcal{A}(\sigma)$ that 
\begin{align*}x\equiv\tfrac{p+1}{2}q\:\text{(mod }p)\Rightarrow\exists j,x-q&=(D(\beta)_j'-j+1)p-\tfrac{p+1}{2}q\\
&=(D(\beta)_j'-j+1-q)p+\tfrac{p-1}{2}q\\
\Rightarrow\exists i,x-q&=(D(\beta)_i-i)p+\tfrac{p-1}{2}q\in\mathcal{A}(\sigma);\end{align*}
hence $\overline{\sigma}_q=\sigma$. 

Conversely, suppose $\sigma\in\overline{C}_q$. Then since $\sigma$ shares the $p$-set of $\Upsilon_{p,q}$, by Lemma \ref{sig2} the $p$-quotient $\mathcal{Q}_p(\sigma)$ must consist of a $q$-bar-core $\beta=\sigma^{(0\text{ mod }p)}$ and two $q$-cores $\delta$ and $\delta'$, where
$$\sigma^{(j\text{ mod }p)}=\begin{cases}\delta&j\equiv iq\text{ mod }p,\:i\in\{1,\dots,\tfrac{p-1}{2}\},\\
\delta'&j\equiv-iq\text{ mod }p,\:i\in\{1,\dots,\tfrac{p-1}{2}\}.\end{cases}$$ 
For each $n\in\{1,\dots,\nicefrac{(p-1)}{2}\}$, we therefore have 
\begin{align*}\{x\in\mathcal{A}(\sigma)|x\equiv nq\text{ (mod }p)\}&=\{(\delta_i-i)p+nq|i\in\mathbb{N}\}\\
\{x\in\mathcal{A}(\sigma)|x\equiv-nq\text{ (mod }p)\}&=\{(\delta_j'-j+1)p-nq|j\in\mathbb{N}\}.\end{align*}
Since $\overline{\sigma}_q=\sigma$, 
\begin{align*}
&\forall k\in\mathbb{N},\beta_kp-q\in\mathcal{A}(\sigma)\Rightarrow\exists j,\beta_k=\delta_j'-j+1;\\
&\forall i\in\mathbb{N},(\delta_i-i)p\in\mathcal{A}(\sigma)\Rightarrow\exists h,\delta_i-i=\beta_h;
\end{align*}
thus $\mathcal{B}^\delta\subseteq\mathcal{A}(\beta)\subseteq\mathcal{B}^{\delta'}$. 
Since $yp-q\in\mathcal{A}(\sigma)$ for each $y\in\mathcal{A}(\beta)\cup\{0\}$, and since $\Delta_{-q\text{ mod }p}\sigma=p-q$, it follows that 
$$\{x\in\mathcal{A}(\sigma)|x\equiv-q\:(\text{mod }p)\}=\{yp-q|y\in\mathcal{A}(\beta)\}\cup\{-q\}=\{(\delta_j'-j+1)p-q|j\in\mathbb{N}\};$$  
hence $\sigma^{(-q\text{ mod }p)}=D(\beta)'$. 
\end{proof}

\begin{exam}
There are 2 possibilitities for $\mu\in\overline{C}_3$ when $\overline{\mu}_5=\Upsilon_{3,5}=(2)$ and $\mu^{(0\text{ mod }5)}=(5,2)$: 
$$\mu=(25,22,19,16,13,10,7,4,1)\text{ and }\mu=(37,34,31,28,25,22,19,16,13,10,7,4,1),$$ 
shown in the second and third configurations below, respectively. 
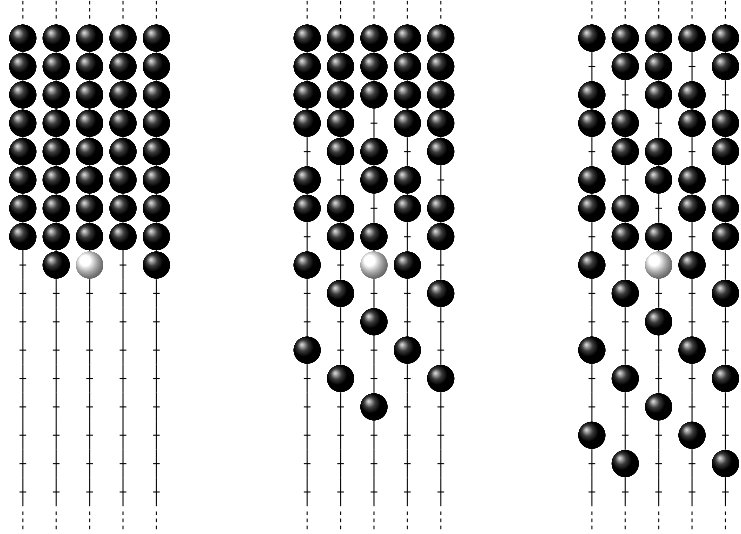
\begin{align*}
&\abacus(vvvvv,bbbbb,bbbbb,bbbbb,bbbbb,bbbbb,bbbbb,bbbbb,bbbbb,nbonb,nnnnn,nnnnn,nnnnn,nnnnn,nnnnn,nnnnn,nnnnn,nnnnn,vvvvv)
&&
\abacus(vvvvv,bbbbb,bbbbb,bbbbb,bbnbb,nbbnb,bnbbn,bbnbb,nbbnb,bnobn,nbnnb,nnbnn,bnnbn,nbnnb,nnbnn,nnnnn,nnnnn,nnnnn,vvvvv)
&&
\abacus(vvvvv,bbbbb,nbbnb,bnbbn,bbnbb,nbbnb,bnbbn,bbnbb,nbbnb,bnobn,nbnnb,nnbnn,bnnbn,nbnnb,nnbnn,bnnbn,nbnnb,nnnnn,vvvvv)
\end{align*}
\end{exam}

\begin{exam}
If $\sigma\in\overline{C}_5$, $\overline{\sigma}_3=\Upsilon_{3,5}=(2)$, and $\sigma^{(0\text{ mod }3)}=(7,4,2)=:\beta$, then $\sigma^{(5\text{ mod }3)}=D(\beta)=(8,6,5,3,2,1,1)$, so $\sigma^{(-5\text{ mod }3)}=D(\beta)'=(7,5,4,3,3,2,1,1)$. 
\begin{align*}
&\Upsilon_{3,5}=(2)&&\sigma=(26,21,17,16,12,11,7,6,2,1)\end{align*}\begin{align*}
&\abacus(vvv,bbb,bbb,bbb,bbb,bbb,bbb,bbb,bbb,bbb,bbn,bon,bnn,nnn,nnn,nnn,nnn,nnn,nnn,nnn,nnn,nnn,vvv)&&
\abacus(vvvvv,bbbbb,bbbbb,bbbbb,bbbbb,bbbbb,bbbbb,bbbbb,bbbbb,bbbbb,bbbbb,nbonb,nnnnn,nnnnn,nnnnn,nnnnn,nnnnn,nnnnn,nnnnn,nnnnn,nnnnn,nnnnn,vvvvv)&&
\abacus(vvv,bbb,bbn,bbb,bnb,bbn,nbb,bnn,bbb,nnb,bbn,nob,bnn,nbb,nnn,bbn,nnb,bnn,nbn,nnn,bnn,nnn,vvv)&&
\abacus(vvvvv,bbbbb,bbbbb,bbbbb,bbbbb,bbbbb,bnbbb,bnbbb,nnbbb,nnbbb,nnbbb,nnobb,nnnbb,nnnbb,nnnbb,nnnbn,nnnbn,nnnnn,nnnnn,nnnnn,nnnnn,nnnnn,vvvvv)
\end{align*}
\end{exam}

Our final result establishes that as a $\mathfrak{W}_p\times\mathfrak{W}_q$ set, $\overline{C}^\Upsilon_{p,q}$ is isomorphic to $(2^{\{1,\dots,\nicefrac{(p-1)}{2}\}}\times\overline{C}_p)\times\overline{C}_q$, where $\mathfrak{W}_p$ and $\mathfrak{W}_q$ act at level 1 on $\overline{C}_p$ and $\overline{C}_q$. 

\begin{cor}
\begin{enumerate}
\item 
There is a bijection 
\begin{align*}\overline{C}^\Upsilon_{p,q}&\to2^{\{1,\dots,\nicefrac{(p-1)}{2}\}}\times\overline{C}_p\times\overline{C}_q\\
\lambda&\mapsto(\{i\in\{1,\dots,\nicefrac{(p-1)}{2}\}|\lambda^{(0\text{ mod }q)}_iq+p\in\mathcal{A}(\lambda)\},\lambda^{(0\text{ mod }q)},\lambda^{(0\text{ mod }p)}). 
\end{align*}
\item Suppose $\lambda\in\mathcal{P}_2$ and $a\in\mathfrak{W}_p$. Then $(a\lambda)^{(0\text{ mod }q)}=a(\lambda^{(0\text{ mod }q)})$, 
where $a$ acts at level $q$ on $\mathcal{P}_2$ and at level 1 on $\overline{C}_p$. 
\end{enumerate}
\end{cor}
\begin{proof}
(1) 
Denote the map by $\Phi$. 
If $\lambda$ belongs to the orbit of $\Upsilon_{p,q}$ under the action of $\mathfrak{W}_p\times\mathfrak{W}_q$, then by Lemma \ref{lem1}, $\lambda$ is a $pq$-bar-core. Thus, by Lemma \ref{lem3}(1), $\lambda^{(0\text{ mod }q)}$ is a $p$-bar-core, and $\lambda^{(0\text{ mod }p)}$ is a $q$-bar-core, so $\Phi$ is well-defined. To show that $\Phi$ is a bijection, we construct an inverse. 

When $X\subseteq\{1,\dots,\nicefrac{(p-1)}{2}\}$, $\alpha\in\overline{C}_p$, and $\beta\in\overline{C}_q$, we let $\mu$ be the bar partition with $\overline{\mu}_q=\Upsilon_{p,q}$, $\mu^{(0\text{ mod }q)}=\alpha$, and $\mu^{(j\text{ mod }q)}=\gamma$ when $j\not\equiv0$ (mod $q$), for some self-conjugate $p$-core partition $\gamma$ such that 
$$\mathcal{B}^\gamma_{\nicefrac{(1-p)}{2}}\subseteq\mathcal{A}(\alpha)\subseteq\mathcal{B}^\gamma_{\nicefrac{(p+1)}{2}}\text{ and }\alpha_iq+p\in\mathcal{A}(\mu),i\in\{1,\dots,\nicefrac{(p-1)}{2}\}\Leftrightarrow i\in X.$$ 
We let $\sigma$ be the bar partition with $\overline{\sigma}_p=\Upsilon_{p,q}$ and 
$$\sigma^{(j\text{ mod }p)}=\begin{cases}\beta&\text{ if } j\equiv0\:(\text{mod }p),\\
D(\beta)&\text{ if } j\equiv iq\:(\text{mod }p)\text{ for some }1\leq i\leq\tfrac{p-1}{2},\\
D(\beta)'&\text{ otherwise.}\end{cases}$$ 
Then $\mu\in\overline{C}_p$ and $\sigma\in\overline{C}_q$ by Proposition \ref{prop9}, so we can define $\Psi(X,\alpha,\beta)$ to be the bar partition $\mu\boxplus\sigma$, which is contained in the orbit of $\Upsilon_{p,q}$ under the action of $\mathfrak{W}_p\times\mathfrak{W}_q$ by Proposition \ref{prop8}. 

Suppose $\lambda\in\overline{C}^\Upsilon_{p,q}$, and let $\Phi(\lambda)=(X,\alpha,\beta)$. We need to show that $\Psi((X,\alpha,\beta))=\lambda$. By Corollary \ref{cor3}, $\overline{\lambda}_q$ has $p$-bar-core $\Upsilon_{p,q}$. 
Since $\lambda$ and $\overline{\lambda}_q$ lie in the same level $q$ orbit of $\mathfrak{W}_p$, they have the same $p$-quotient up to reordering. 
Thus $(\overline{\lambda}_q)^{(0\text{ mod }p)}=\beta=\lambda^{(0\text{ mod }p)}$, and $\overline{\lambda}_q$ is the unique $q$-bar-core $\sigma$ with $p$-bar-core $\Upsilon_{p,q}$ and $p$-quotient consisting of $\beta$, $D(\beta)$, and $D(\beta)'$ (by Proposition \ref{prop9}). 
Similarly, $\overline{\lambda}_p$ is one of $2^{\nicefrac{(p-1)}{2}}$ bar partitions with $q$-bar-core $\Upsilon_{p,q}$ and $q$-quotient $\mathcal{Q}_q(\lambda)$ consisting of $\alpha\in\overline{C}_p$ and a self-conjugate $p$-core $\gamma$ such that $\mathcal{B}^\gamma_{\nicefrac{(1-p)}{2}}\subseteq\mathcal{A}(\alpha)\subseteq\mathcal{B}^\gamma_{\nicefrac{(p+1)}{2}}$. If we denote by $\mu$ the unique such bar partition with $\alpha_iq+p\in\mathcal{A}(\mu),i\in\{1,\dots,\nicefrac{(p-1)}{2}\}\Leftrightarrow i\in X$, then $\Psi((X,\alpha,\beta))=\mu\boxplus\sigma=\lambda$. 

Finally, let $(X,\alpha,\beta)\in2^{\{1,\dots,\nicefrac{(p-1)}{2}\}}\times\overline{C}_p\times\overline{C}_q$, so that $\Psi(X,\alpha,\beta)=\mu\boxplus\sigma$, where $\mu$ is the bar partition with $\overline{\mu}_q=\Upsilon_{p,q}$, $\mu^{(0\text{ mod }q)}=\alpha$, and $\mu^{(j\text{ mod }q)}=\gamma$ when $j\not\equiv0$ (mod $q$), for some self-conjugate $p$-core partition $\gamma$ such that 
$$\mathcal{B}^\gamma_{\nicefrac{(1-p)}{2}}\subseteq\mathcal{A}(\alpha)\subseteq\mathcal{B}^\gamma_{\nicefrac{(p+1)}{2}}\text{ and }\alpha_iq+p\in\mathcal{A}(\mu),i\in\{1,\dots,\nicefrac{(p-1)}{2}\}\Leftrightarrow i\in X,$$ 
and $\sigma$ is the bar partition with $\overline{\sigma}_p=\Upsilon_{p,q}$ and 
$$\sigma^{(j\text{ mod }p)}=\begin{cases}\beta&\text{ if } j\equiv0\:(\text{mod }p),\\
D(\beta)&\text{ if } j\equiv iq\:(\text{mod }p)\text{ for some }1\leq i\leq\tfrac{p-1}{2},\\
D(\beta)'&\text{ otherwise.}\end{cases}$$ 
Then since $\mu\boxplus\sigma$ shares a $q$-quotient with $\mu$ and has the same $p$-quotient as $\sigma$ up to reordering, we find that $\Phi(\mu\boxplus\sigma)=(X,\alpha,\beta)$. Hence $\Phi$ and $\Psi$ are mutual inverses, and thus bijections. 

\vspace{3mm}

(2) Let $\rho=\lambda^{(0\text{ mod }q)}$, so that $\mathcal{A}(\lambda)\cap q\mathbb{Z}=\{xq|x\in\mathcal{A}(\rho)\}$. Since $xq\equiv iq$ (mod $p)$ $\Leftrightarrow$ $x\equiv i$ (mod $p$), for each $i\in\{0,\dots,\nicefrac{(p-1)}{2}\}$, by the definition of the $\mathfrak{W}_p$-action we have  
$$\mathcal{A}(\delta_i\lambda)\cap q\mathbb{Z}=\delta_i\mathcal{A}(\lambda)\cap q\mathbb{Z}=\{(\delta_ix)q|x\in\mathcal{A}(\rho)\}=\{xq|x\in\mathcal{A}(\delta_i\rho)\},$$ 
thus $(\delta_i\lambda)^{(0\text{ mod }q)}=\delta_i(\lambda^{(0\text{ mod }q)})$ for each $i$, where $\delta_i$ acts at level $q$ on $\lambda$, and at level 1 on $\rho$. The result follows for all $a\in\mathfrak{W}_p$. 
\end{proof}

\begin{exam}
We calculate part of $\overline{C}^\Upsilon_{3,5}$, the orbit of $\Upsilon_{3,5}=(2)$ under the action of $\mathfrak{W}_3\times\mathfrak{W}_5$, to illustrate the bijection between this orbit and the set $2^{\{1\}}\times\overline{C}_3\times\overline{C}_5$. The top part of Figure \ref{bij1} below shows the level 1 action of $\mathfrak{W}_5=\langle\delta_0,\delta_1,\delta_2\rangle$ on $\overline{C}_5$, and on the left, the level 1 action of $\mathfrak{W}_3=\langle\epsilon_0,\epsilon_1\rangle$ on $\overline{C}_3$: 
\begin{align*} 
&\delta_0x=\begin{cases}x-2&x\equiv1\:(\text{mod }5),\\
x+2&x\equiv4\:(\text{mod }5),\\
x&\text{otherwise;}\end{cases}
&&\delta_1x=\begin{cases}x-1&x\equiv2,4\:(\text{mod }5),\\
x+1&x\equiv1,3\:(\text{mod }5),\\
x&\text{otherwise;}\end{cases}\\
&\delta_2x=\begin{cases}x-1&x\equiv3\:(\text{mod }5),\\
x+1&x\equiv2\:(\text{mod }5),\\
x&\text{otherwise;}\end{cases}\\
&\epsilon_0x=\begin{cases}x-2&x\equiv1\:(\text{mod }3),\\
x+2&x\equiv2\:(\text{mod }3),\\
x&\text{otherwise;}\end{cases}
&&\epsilon_1x=\begin{cases}x-1&x\equiv2\:(\text{mod }3),\\
x+1&x\equiv1\:(\text{mod }3),\\
x&\text{otherwise.}\end{cases}
\end{align*}
The edges in the main part of the diagram represent the level 3 action of the generators $\delta_0,\delta_1,\delta_2$ of $\mathfrak{W}_5$, and the level 5 action of the generators $\epsilon_0,\epsilon_1$ of $\mathfrak{W}_3$: 
\begin{align*} 
&\delta_0x=\begin{cases}x-6&x\equiv3\:(\text{mod }5),\\
x+6&x\equiv2\:(\text{mod }5),\\
x&\text{otherwise;}\end{cases}
&&\delta_1x=\begin{cases}x-3&x\equiv1,2\:(\text{mod }5),\\
x+3&x\equiv3,4\:(\text{mod }5),\\
x&\text{otherwise;}\end{cases}\\
&\delta_2x=\begin{cases}x-3&x\equiv4\:(\text{mod }5),\\
x+3&x\equiv1\:(\text{mod }5),\\
x&\text{otherwise;}\end{cases}\\
&\epsilon_0x=\begin{cases}x-10&x\equiv2\:(\text{mod }3),\\
x+10&x\equiv1\:(\text{mod }3),\\
x&\text{otherwise;}\end{cases}
&&\epsilon_1x=\begin{cases}x-5&x\equiv1\:(\text{mod }3),\\
x+5&x\equiv2\:(\text{mod }3),\\
x&\text{otherwise.}\end{cases}
\end{align*}

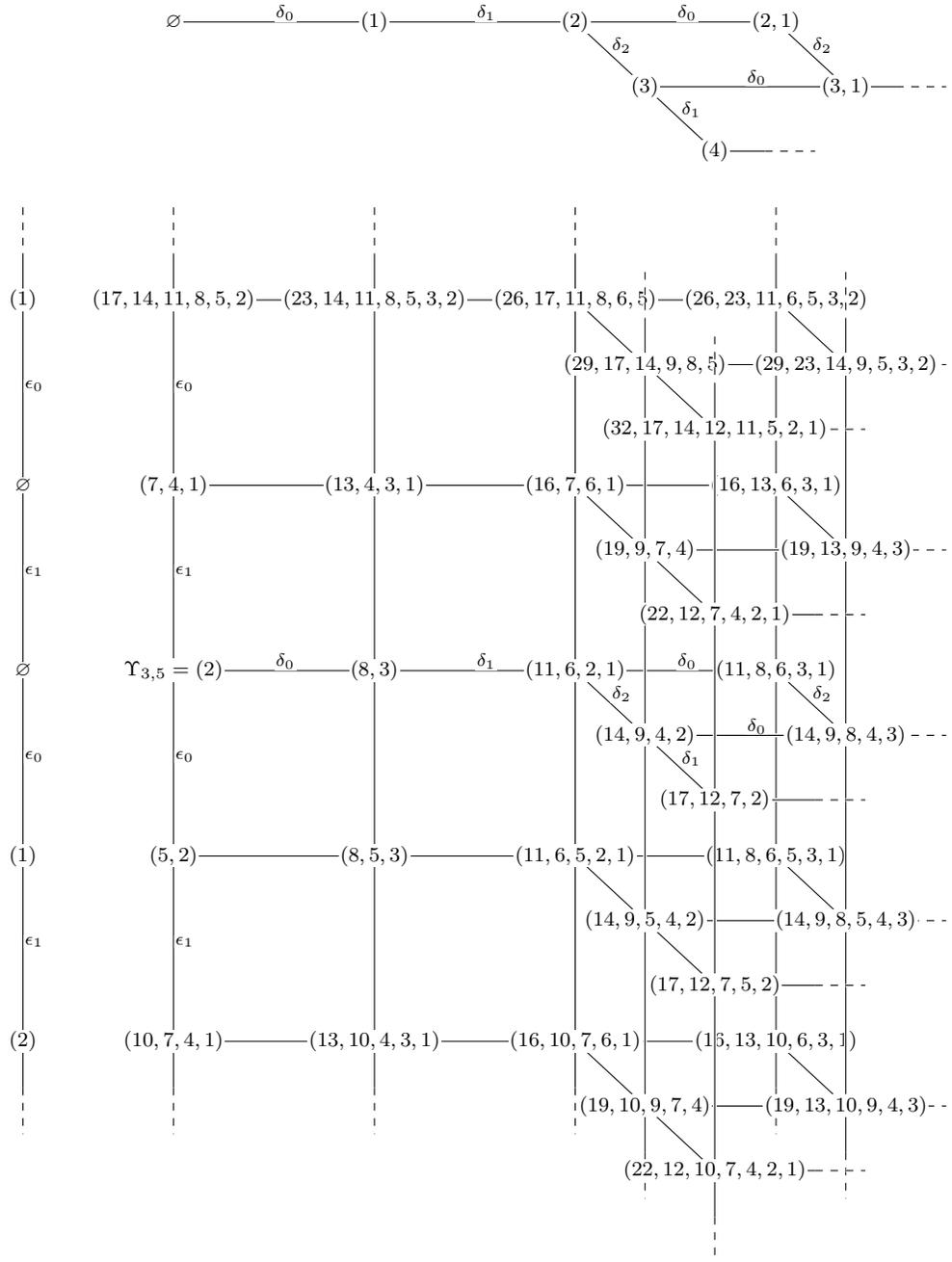
\begin{figure}[p]
\newcommand\drawbw[1]{\draw[line width=1.5pt,white]{#1};\draw{#1};}
\newcommand\drawbwd[1]{\draw[line width=1.5pt,white,dashed]{#1};\draw[dashed]{#1};}
{\footnotesize\[\tdplotsetmaincoords{67.5}{0}
\begin{tikzpicture}[tdplot_main_coords,scale=0.041]
\tikzstyle{every node}=[fill=white,inner sep=0.5pt]
\coordinate(bb)at(0.825cm,-2cm,0cm);
\coordinate(aa)at(2.375cm,0cm,0cm);
\coordinate(cc)at(0cm,0cm,-2.375cm);
\coordinate(sb)at($.5*(bb)$);
\coordinate(sa)at($.5*(aa)$);
\coordinate(sc)at($.5*(cc)$);
\coordinate(qc)at($.25*(cc)$);
\draw($-.75*(aa)-.25*(cc)$)--++(qc)node{$(1)$}--++(sc)node[midway,above right,pos=1]{\scriptsize$\epsilon_0$}--++(sc)node{$\varnothing$}--++(sc)node[midway,above right,pos=1]{\scriptsize$\epsilon_1$}--++(sc)node{$\varnothing$}--++(sc)node[midway,above right,pos=1]{\scriptsize$\epsilon_0$}--++(sc)node{$(1)$}--++(sc)node[midway,above right,pos=1]{\scriptsize$\epsilon_1$}--++(sc)node{$(2)$}--++($.25*(cc)$);
\draw[dashed]($-.75*(aa)-.5*(cc)$)--++($.25*(cc)$);
\draw[dashed]($-.75*(aa)+4.25*(cc)$)--++($.25*(cc)$);
\foreach\x in {-1.5}{
\drawbw{($\x*(cc)$)--++(sa)node[midway,above right,pos=1]{\scriptsize$\delta_0$}--++(sa)--++(sa)node[midway,above right,pos=1]{\scriptsize$\delta_1$}--++(sa)--++(sa)node[midway,above right,pos=1]{\scriptsize$\delta_0$}--++(sa)}
\drawbw{($\x*(cc)+2*(aa)$)--++(sb)node[midway,above right,pos=1]{\scriptsize$\delta_2$}--++(sb)}
\drawbw{($\x*(cc)+3*(aa)$)--++(sb)node[midway,above right,pos=1]{\scriptsize$\delta_2$}--++(sb)}
\draw($\x*(cc)$)node{$\varnothing$}++(sa)++(sa)node{$(1)$}++(sa)++(sa)node{$(2)$}++(sa)++(sa)node{$(2,1)$};
\drawbw{($\x*(cc)+2*(aa)+(bb)$)--++(sa)node[midway,above right,pos=1]{\scriptsize$\delta_0$}--++(sa)--++($.25*(aa)$)}
\drawbwd{($\x*(cc)+2.25*(aa)+2*(bb)$)--++($.25*(aa)$)}
\drawbw{($\x*(cc)+2*(aa)+(bb)$)--++(sb)node[midway,above right,pos=1]{\scriptsize$\delta_1$}--++(sb)}
\draw($\x*(cc)+2*(aa)+(bb)$)node{$(3)$}++(sa)++(sa)node{$(3,1)$};
\drawbw{($\x*(cc)+2*(aa)+2*(bb)$)--++($.25*(aa)$)}
\drawbwd{($\x*(cc)+3.25*(aa)+(bb)$)--++($.25*(aa)$)}
\draw($\x*(cc)+2*(aa)+2*(bb)$)node{$(4)$};
}
\foreach\x in{0,1,2,3}{
\drawbwd{($\x*(aa)$)++($4.25*(cc)$)--++($.25*(cc)$)}
\drawbwd{($\x*(aa)-.5*(cc)$)--++($.25*(cc)$)}
}
\foreach\x in{1,2,3}{
\drawbw{($\x*(aa)-.25*(cc)$)--++($4.5*(cc)$)}
}
\drawbw{($-.25*(cc)$)--++(qc)--++(sc)node[midway,above right,pos=1]{\scriptsize$\epsilon_0$}--++(cc)node[midway,above right,pos=1]{\scriptsize$\epsilon_1$}--++(cc)node[midway,above right,pos=1]{\scriptsize$\epsilon_0$}--++(cc)node[midway,above right,pos=1]{\scriptsize$\epsilon_1$}--++($.75*(cc)$)}
\foreach\x in{0,1,3,4}{
\drawbw{($\x*(cc)$)--++($3*(aa)$)}
}
\drawbw{($2*(cc)$)--++(sa)node[midway,above right,pos=1]{\scriptsize$\delta_0$}--++(sa)--++(sa)node[midway,above right,pos=1]{\scriptsize$\delta_1$}--++(sa)--++(sa)node[midway,above right,pos=1]{\scriptsize$\delta_0$}--++(sa)}
\foreach\x in{2,3}{
\foreach\y in{0,1,3,4}{
\drawbw{($\y*(cc)+\x*(aa)$)--++(bb)}
}
}
\drawbw{($2*(aa)+2*(cc)$)--++(sb)node[midway,above right,pos=1]{\scriptsize$\delta_2$}--++(sb)}
\drawbw{($3*(aa)+2*(cc)$)--++(sb)node[midway,above right,pos=1]{\scriptsize$\delta_2$}--++(sb)}
\draw(0,0,0)node{$(17,14,11,8,5,2)$}++(aa)node{$(23,14,11,8,5,3,2)$}++(aa)node{$(26,17,11,8,6,5)$}++(aa)node{$(26,23,11,6,5,3,2)$};
\draw(cc)node{$(7,4,1)$}++(aa)node{$(13,4,3,1)$}++(aa)node{$(16,7,6,1)$}++(aa)node{$(16,13,6,3,1)$};
\draw($2*(cc)$)node{$\Upsilon_{3,5}=(2)$}++(aa)node{$(8,3)$}++(aa)node{$(11,6,2,1)$}++(aa)node{$(11,8,6,3,1)$};
\draw($3*(cc)$)node{$(5,2)$}++(aa)node{$(8,5,3)$}++(aa)node{$(11,6,5,2,1)$}++(aa)node{$(11,8,6,5,3,1)$};
\draw($4*(cc)$)node{$(10,7,4,1)$}++(aa)node{$(13,10,4,3,1)$}++(aa)node{$(16,10,7,6,1)$}++(aa)node{$(16,13,10,6,3,1)$};
\foreach\x in{0,1,2,3,4}{
\drawbwd{($\x*(cc)+3.25*(aa)+(bb)$)--++($.25*(aa)$)}
}
\foreach\x in{0,1,3,4}{
\drawbw{($\x*(cc)+2*(aa)+(bb)$)--++($1.25*(aa)$)}
}
\drawbw{($2*(aa)+(bb)+2*(cc)$)--++(sa)node[midway,above right,pos=1]{\scriptsize$\delta_0$}--++(sa)--++($.25*(aa)$)}
\foreach\x in{2,3}{
\drawbwd{($\x*(aa)+(bb)+4.25*(cc)$)--++($.25*(cc)$)}
\drawbwd{($\x*(aa)+(bb)-.5*(cc)$)--++($.25*(cc)$)}
\drawbw{($\x*(aa)+(bb)-.25*(cc)$)--++($4.5*(cc)$)}
}
\foreach\x in{0,1,3,4}{
\drawbw{($\x*(cc)+2*(aa)+(bb)$)--++(bb)}
}
\drawbw{($2*(aa)+(bb)+2*(cc)$)--++(sb)node[midway,above right,pos=1]{\scriptsize$\delta_1$}--++(sb)}
\draw(aa)++(aa)++(bb)node{$(29,17,14,9,8,5)$}++(cc)node{$(19,9,7,4)$}++(cc)node{$(14,9,4,2)$}++(cc)node{$(14,9,5,4,2)$}++(cc)node{$(19,10,9,7,4)$};
\draw(aa)++(aa)++(aa)++(bb)node{$(29,23,14,9,5,3,2)$}++(cc)node{$(19,13,9,4,3)$}++(cc)node{$(14,9,8,4,3)$}++(cc)node{$(14,9,8,5,4,3)$}++(cc)node{$(19,13,10,9,4,3)$};
\drawbwd{($2*(aa)+2*(bb)-.5*(cc)$)--++($.25*(cc)$)}
\drawbwd{($2*(aa)+2*(bb)+4.25*(cc)$)--++($.25*(cc)$)}
\drawbw{($2*(aa)+2*(bb)-.25*(cc)$)--++($4.5*(cc)$)}
\foreach\x in{0,1,2,3,4}{
\drawbwd{($\x*(cc)+2.5*(aa)+2*(bb)$)--++($.25*(aa)$)}
\drawbw{($\x*(cc)+2*(aa)+2*(bb)$)--++($.5*(aa)$)}
}
\draw(aa)++(aa)++(bb)++(bb)node{$(32,17,14,12,11,5,2,1)$}++(cc)node{$(22,12,7,4,2,1)$}++(cc)node{$(17,12,7,2)$}++(cc)node{$(17,12,7,5,2)$}++(cc)node{$(22,12,10,7,4,2,1)$};
\end{tikzpicture}
\]}
\caption{The bijection between $\overline{C}^\Upsilon_{3,5}$ and $2^{\{1\}}\times\overline{C}_3\times\overline{C}_5$}
\label{bij1}
\end{figure}

Below are some abacus displays to illustrate the bijection $\overline{C}^\Upsilon_{3,5}\to2^{\{1\}}\times\overline{C}_3\times\overline{C}_5$: 
\begin{align*}
&(X\in2^{\{-2\}},\varnothing,\varnothing):&&\Upsilon_{3,5}=(2)&&(7,4,1)\\
&&&\abacus(vvvvv,bbbbb,bbbbb,nbonb,nnnnn,nnnnn,vvvvv)&&\abacus(vvvvv,bbbbb,nbbnb,bnobn,nbnnb,nnnnn,vvvvv)\\
&(X\in2^{\{8\}},(1),(2)):&&(11,6,5,2,1)&&(26,17,11,8,6,5)\\
&&&\abacus(vvvvv,bbbbb,bbbbb,bbbbb,bbbbb,bnbbb,bnnbb,nnobb,nnbbn,nnnbn,nnnnn,nnnnn,nnnnn,nnnnn,vvvvv)&&\abacus(vvvvv,bbbbb,bnbbb,bbbbb,nbbbb,bnbbn,bnnbb,bbonn,nnbbn,bnnbn,nnnnb,nnnnn,nnnbn,nnnnn,vvvvv)
\end{align*}
The following abacus displays illustrate the bijection $\overline{C}^\Upsilon_{5,3}\to2^{\{1,2\}}\times\overline{C}_5\times\overline{C}_3$: 
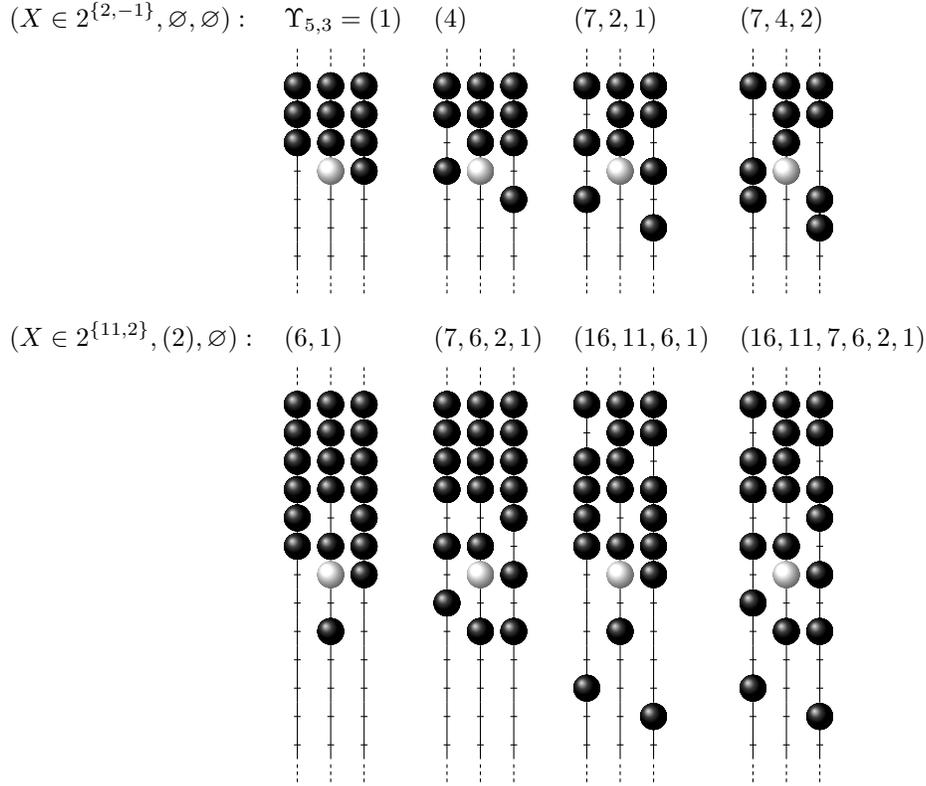
\begin{align*}
&(X\in2^{\{2,-1\}},\varnothing,\varnothing):&&\Upsilon_{5,3}=(1)&&(4)&&(7,2,1)&&(7,4,2)\\
&&&\abacus(vvv,bbb,bbb,bbb,nob,nnn,nnn,nnn,vvv)&&\abacus(vvv,bbb,bbb,nbb,bon,nnb,nnn,nnn,vvv)&&\abacus(vvv,bbb,nbb,bbn,nob,bnn,nnb,nnn,vvv)&&\abacus(vvv,bbb,nbb,nbn,bon,bnb,nnb,nnn,vvv)\\
&(X\in2^{\{11,2\}},(2),\varnothing):&&(6,1)&&(7,6,2,1)&&(16,11,6,1)&&(16,11,7,6,2,1)\\
&&&\abacus(vvv,bbb,bbb,bbb,bbb,bnb,bbb,nob,nnn,nbn,nnn,nnn,nnn,nnn,vvv)&&\abacus(vvv,bbb,bbb,bbb,bbb,nnb,bbn,nob,bnn,nbb,nnn,nnn,nnn,nnn,vvv)&&\abacus(vvv,bbb,nbb,bbn,bbb,bnb,bbb,nob,nnn,nbn,nnn,bnn,nnb,nnn,vvv)&&\abacus(vvv,bbb,nbb,bbn,bbb,nnb,bbn,nob,bnn,nbb,nnn,bnn,nnb,nnn,vvv)\end{align*}

The orbit of $\Upsilon_{5,3}=(1)$ under the action of $\mathfrak{W}_3\times\mathfrak{W}_5$ is in bijection with the set $2^{\{1,2\}}\times\overline{C}_3\times\overline{C}_5$. In Figure \ref{bij2}, we illustrate part of this orbit; the edges in this diagram represent the level 3 action of $\mathfrak{W}_5=\langle\delta_0,\delta_1,\delta_2\rangle$, and for each $\lambda\in\mathcal{P}_2$ in the diagram, $\lambda^{(0\text{ mod }5)}=\varnothing$. 

\begin{figure}[p]
\newcommand\drawbw[1]{\draw[line width=1.5pt,white]{#1};\draw{#1};}
\newcommand\drawbwd[1]{\draw[line width=1.5pt,white,dashed]{#1};\draw[dashed]{#1};}
{\footnotesize\[\tdplotsetmaincoords{67.5}{0}
\begin{tikzpicture}[tdplot_main_coords,scale=0.041]
\tikzstyle{every node}=[fill=white,inner sep=0.5pt]
\coordinate(bb)at(0.825cm,-2cm,0cm);
\coordinate(aa)at(2.375cm,0cm,0cm);
\coordinate(cc)at(0cm,0cm,-2.375cm);
\coordinate(sb)at($.5*(bb)$);
\coordinate(sa)at($.5*(aa)$);
\coordinate(sc)at($.5*(cc)$);
\coordinate(la)at(2.5cm,0cm,0cm);
\coordinate(na)at($.445*(aa)$);
\coordinate(ma)at($.555*(aa)$);
\drawbw{(0,0,0)--++(sc)node[midway,above right,pos=1]{\scriptsize$\delta_2$}--++(cc)node[midway,above right,pos=1]{\scriptsize$\delta_1$}--++(cc)node[midway,above right,pos=1]{\scriptsize$\delta_2$}--++(sc)}
\foreach\x in{0,2}{
\drawbw{($(aa)+\x*(cc)$)--++(sc)node[midway,above right,pos=1]{\scriptsize$\delta_2$}--++(sc)}
}
\foreach\x in{0,1,2,3}{
\drawbw{($\x*(cc)$)--++(sa)node[midway,above right,pos=1]{\scriptsize$\delta_0$}--++(sa)--++(sa)node[midway,above right,pos=1]{\scriptsize$\delta_1$}--++(sa)--++(na)node[midway,above right,pos=1]{\scriptsize$\delta_0$}--++(ma)}
}
\foreach\x in{2,3}{
\foreach\y in{0,1,2,3}{
\drawbw{($\y*(cc)+\x*(aa)$)--++(sb)node[midway,above right,pos=1]{\scriptsize$\delta_2$}--++(sb)}
}
}
\draw(0,0,0)node{$\Upsilon_{5,3}=(1)$}++(aa)node{$(3,1)$}++(aa)node{$(6,1)$}++(aa)node{$(6,3,1)$};
\draw(cc)node{$(4)$}++(aa)node{$(4,3)$}++(aa)node{$(7,6,2,1)$}++(aa)node{$(13,8,6,3,1)$};
\draw($2*(cc)$)node{$(7,2,1)$}++(aa)node{$(13,8,3,1)$}++(aa)node{$(16,11,6,1)$}++(aa)node{$(16,11,6,3,1)$};
\draw($3*(cc)$)node{$(7,4,2)$}++(aa)node{$(13,8,4,3)$}++(aa)node{$(16,11,7,6,2,1)$}++(la)node{$(16,13,11,8,6,3,1)$};
\foreach\x in{0,1,2,3}{
\drawbwd{($\x*(cc)+3.25*(aa)+(bb)$)--++($.25*(aa)$)}
}
\foreach\x in{0,1,2,3}{
\drawbw{($\x*(cc)+2*(aa)+(bb)$)--++(na)node[midway,above right,pos=1]{\scriptsize$\delta_0$}--++(ma)--++($.25*(aa)$)}
}
\foreach\x in{0,1,2,3}{
\drawbw{($\x*(cc)+2*(aa)+(bb)$)--++(sb)node[midway,above right,pos=1]{\scriptsize$\delta_1$}--++(sb)}
}
\draw(aa)++(aa)++(bb)node{$(9,4)$}++(cc)node{$(9,7,4,2)$}++(cc)node{$(19,14,9,4)$}++(cc)node{$(19,14,13,9,8,4,3)$};
\draw(aa)++(aa)++(aa)++(bb)node{$(9,4,3)$}++(cc)node{$(13,9,8,4,3)$}++(cc)node{$(19,14,9,4,3)$}++(cc)node{$(19,14,13,9,8,4,3)$};
\foreach\x in{0,1,2,3}{
\drawbwd{($\x*(cc)+2.5*(aa)+2*(bb)$)--++($.25*(aa)$)}
\drawbw{($\x*(cc)+2*(aa)+2*(bb)$)--++($.5*(aa)$)}
}
\draw(aa)++(aa)++(bb)++(bb)node{$(12,7,2,1)$}++(cc)node{$(12,7,4,2)$}++(cc)node{$(22,17,12,7,2,1)$}++(cc)node{$(22,17,12,7,4,2)$};
\end{tikzpicture}
\]}
\caption{The branch of $\overline{C}^\Upsilon_{5,3}$ corresponding to the 3-bar-core $\varnothing$}
\label{bij2}
\end{figure}
\end{exam}

\bibliographystyle{amsplain-ac}
\bibliography{researchbib}
\end{document}